\documentclass[12pt]{article}
\usepackage{style}
\title{MIRROR SYMMETRY AND FUKAYA CATEGORIES OF SINGULAR HYPERSURFACES}
\date{\today}

\author{Maxim Jeffs}

\begin{document}

\maketitle

\begin{abstract}
We consider a definition of the Fukaya category of a singular hypersurface proposed by Auroux, given by localizing the Fukaya category of a nearby fiber at Seidel's natural transformation, and show that this possesses several desirable properties. Firstly, we prove an A-side analog of Orlov's derived Kn\"orrer periodicity theorem by showing that Auroux's category is derived equivalent to the Fukaya-Seidel category of a higher-dimensional Landau-Ginzburg model. Secondly, we describe how this definition implies homological mirror symmetry for some large complex structure limit degenerations of abelian varieties.
\end{abstract}

\tableofcontents

\section{Introduction}

At a basic level, homological mirror symmetry (HMS) conjectures a relationship between the Fukaya category of a K\"ahler manifold $Y$ and the category of coherent sheaves on a `mirror' K\"ahler manifold $\check{Y}$ of the same dimension. This is also expected to hold in the reverse direction, relating the Fukaya category of the mirror $\check{Y}$ and the coherent sheaves on the original K\"ahler manifold $Y$. In many natural instances of mirror symmetry (such as when $Y$ is non-compact), the `intrinsic' mirror $\check{Y}$ is often expected to be a singular variety. However, a principled definition of the Fukaya category of $\check{Y}$ is lacking in this case. Some of the most important instances, discussed in detail in \S \ref{sec:pants}, are the higher-dimensional pairs of pants $\Pi_n = \set{x_1 + \cdots + x_{n+1} + 1 =0} \subset (\CC^{\ast})^{n+1}$, whose naturally-constructed mirrors are given by the singular hypersurfaces $\set{z_1 \cdots z_{n+1} =0 } \subset \CC^{n+1}$. Using tropical techniques, it is possible to decompose other types of non-compact K\"ahler manifolds into products of pairs of pants of various dimensions; thus understanding mirror symmetry in such cases is of crucial importance for approaches to proving HMS that rely on gluing techniques. 

Moreover, one of the principal desiderata for any approach to constructing a mirror space is that the mirror of the mirror should be the original K\"ahler manifold; in other words, HMS should apply in both directions. However, this immediately runs into the problem that many `intrinsic' mirror constructions naturally produce singular varieties, and without an understanding of their Fukaya categories, it is difficult to proceed further. While incredible progress has been made understanding mirror symmetry for smooth K\"ahler manifolds, outside of the orbifold case little has been written about mirror symmetry for the $A$-model of singular varieties: discussion is included in \S \ref{sec:ms}.

In the case of singular hypersurfaces, the key insight of Auroux is that, while the symplectic geometry of the hypersurface itself may be difficult to understand intrinsically, given a smoothing of this hypersurface, the nearby fibers are perfectly good smooth symplectic manifolds, and their Fukaya categories come equipped with extra algebraic data: a set of morphisms coming from counts of certain holomorphic sections which capture which parts of the manifold degenerate when passing to the singular fiber. In accordance with the philosophy of perverse schobers \cite{PervI}, the invariant cycles theorem suggests that performing the categorical localization at these morphisms would therefore yield the `Fukaya category' of the singular hypersurface. As we shall illustrate in \S \ref{sec:cis}, this technique of localization at algebraic monodromy data has a broader domain of geometric applicability. A precise version of Auroux's definition will be given in \S \ref{sec:cup}.

To summarize this definition, let $X$ be a Stein manifold. For a singular symplectic fibration $f:X \to \CC$ with a single singular fiber over $0$, Seidel \cite{SHasHH,LFI} defines a natural transformation $s: \mu \to \mathrm{id}$ where $\mu$ is the clockwise monodromy functor acting on the wrapped Fukaya category $\scr{W}(f\inv(t))$ of the general fiber (see \S \ref{sec:defns} for our definition). The following definition for the wrapped Fukaya category of the singular fiber was proposed by Auroux:

\begin{restatable}{definition}{auroux}(Auroux)\label{defn:auroux} Suppose a singular symplectic fibration $f:X \to \CC$ has precisely one singular fiber, over $0$. Then the wrapped Fukaya category of $f\inv(0)$ is defined to be the localization of the wrapped Fukaya category of a nearby fiber $f\inv(t)$ for sufficiently small $t\neq 0$ at the natural transformation $s: \mu \to \mathrm{id}$:
\begin{equation*}
    D\scr{W}(f\inv(0)) = D\scr{W}(f\inv(t))[s\inv]
.\end{equation*}
\end{restatable}

Here this is the localization of $A_{\infty}$-categories in the sense of \cite{LO} and $D$ denotes the category of twisted complexes. The definition can also be extended to the case of monotone fibers, or any situation where Seidel's natural transformation is defined.

In this paper, we shall prove several results that illustrate why this is indeed the correct definition of the Fukaya category of a singular hypersurface: an A-side analog of Orlov's derived Kn\"orrer periodicity theorem (Theorem \ref{thm:main_thm}), and several homological mirror symmetry equivalences at the large complex structure limit (Theorems \ref{thm:new_theorem} and \ref{thm:lcsl}), described in detail in the following sections.

We shall give a heuristic description of how Definition \ref{defn:auroux} has a very natural relation to a version of Biran-Cornea's Lagrangian cobordism groups in \S \ref{sec:cobs}. There is also a natural analog of the `nearby cycles' functor one can construct that restricts Lagrangians from the wrapped Fukaya category of a punctured neighbourhood of the singular fiber to the wrapped Fukaya category of the singular fiber itself (cf. \cite{Speculations}). 

\subsection{Kn\"orrer Periodicity} \label{sec:cis}

The first theorem is inspired by the derived Kn\"orrer periodicity theorem of Orlov \cite[Corollary 3.2]{Orlov}, which says that if $X$ is a smooth quasi-projective variety, and $f:X \to \CC$ is a regular function with $f\inv(0)$ smooth, then there is an equivalence of categories
\begin{equation*}
    D^b \mathrm{Coh}(f\inv(0)) \to \mathrm{Sing}(X \times \CC, z f)
\end{equation*}
where $z$ is the coordinate on $\CC$.

Though \cite{Orlov} stated this theorem in the case where $f\inv(0)$ is smooth, the result holds generally for any hypersurface (see \cite[Theorem 1.2]{Hirano}): it hence allows us to study such singular hypersurfaces in terms of an LG model on the smooth variety $X \times \CC$. It was conjectured by Orlov in \cite{Orlov} that the same relation should hold for the $A$-model: by analogy with Orlov's result, we might expect the wrapped Fukaya category of $f\inv(0)$ to be equivalent to the (fiberwise-wrapped) Fukaya-Seidel category of the Landau-Ginzburg model $(X \times \CC, z f)$, denoted $\scr{W}(X \times \CC, zf)$. This is our main result. 

\begin{restatable}{theorem}{mainthm}(Derived Kn\"orrer Periodicity)\label{thm:main_thm}
Suppose that $X$ is a smooth affine variety with an embedding $X \to \CC^N$ inducing a Stein structure on $X$, and suppose $f:X \to \CC$ is the restriction of a polynomial function on $\CC^N$. If $f$ has a single critical fiber $f\inv(0)$ then for sufficiently small $t\neq 0$ there is a quasiequivalence of $A_{\infty}$-categories
\begin{equation*}
    D^{\pi}\scr{W}(f\inv(t))[s\inv] \to D^{\pi}\scr{W}(X \times \CC, z f)
.\end{equation*}
\end{restatable}

Here $D^{\pi}$ is used to denote the idempotent-completion of the twisted complexes; we expect that this result does not hold without taking idempotent-completions. The geometric hypotheses could almost certainly be weakened, but they suffice for applications to mirror symmetry. 

Note that this is a \textit{different} form of `suspension' or `periodicity' for an LG model from that usually considered in singularity theory (compare \cite{suspending}). 

The results above need not be limited to singular hypersurfaces. For singular complete intersections, we can make the definition:

\begin{definition}\label{defn:ci}
Suppose $f_1, \dots, f_k: X \to \CC$ are holomorphic functions on a Stein manifold such that for $t_1, \dots, t_k \neq 0$, the intersection $f_{1}\inv(t_1) \cap \cdots \cap f_{k}\inv(t_k)$ is smooth. Then we define
\begin{equation*}
    D\scr{W}(f_{1}\inv(0) \cap \cdots \cap f_{k}\inv(0)) = D\scr{W}(f_{1}\inv(t_1) \cap \cdots \cap f_{k}\inv(t_k))[s_{1}\inv, \dots, s_{k}\inv]
.\end{equation*}
where $s_{i}$ are the natural transformations coming from the monodromy of $f_i$ around $t_i=0$, and $t_i$ are sufficiently small.
\end{definition}

By iterating our proof, we also expect the main theorem to admit a simple generalization to singular complete intersections:

\begin{conjecture}\label{conj:cis}
Under appropriate hypotheses on $f_1, \dots, f_k$, we have a  quasiequivalence of $A_{\infty}$-categories:
\begin{equation*}
    D^{\pi}\scr{W}(f_{1}\inv(0) \cap \cdots \cap f_{k}\inv(0)) \simeq D^{\pi}\scr{W}(X \times \CC^k, z_1 f_1 + \cdots + z_k f_k)
\end{equation*}
where $z_1, \dots, z_k$ are coordinates on $\CC^k$. 
\end{conjecture}

We shall discuss below how this conjecture could be proved by an extension of the results in this paper.

\subsection{Mirror Symmetry} \label{sec:ms}

The definition of the Fukaya category of a singular hypersurface given above manifestly depends on a choice of smoothing. This is not only desirable, but in fact crucial, for homological mirror symmetry purposes. Classically, mirror symmetry is a relation between a K\"ahler manifold and a large complex structure limit (LCSL) family of complex manifolds. Homological mirror symmetry is expected to be an involution, so it is important to have a notion of the Fukaya category of the singular fiber of this family. Given that the choice of the mirror depends on the entire degeneration, it is not surprising that the definition of the Fukaya category of the singular fiber should involve the data of the smoothing. 


On the other hand, since the germ of an isolated hypersurface singularity has a smooth and connected versal deformation space \cite[p. 144]{KM}, a simple homotopy argument shows that the Fukaya category (of the germ) is independent of the choice of smoothing in this case. Hence it provides a (potentially interesting) symplectic invariant of isolated hypersurface singularities. In general, extra data must be provided for the Fukaya category to be uniquely specified: examples are provided in \S \ref{sec:hms}. Expectations from mirror symmetry suggest that this extra data should take the form of a \textit{log structure} on $f\inv(0)$, since in good cases this is expected to determine a smoothing of $f\inv(0)$. It may be possible to formulate an intrinsic construction of the wrapped Fukaya category of a singular hypersurface with a log structure using Parker's theory of holomorphic curves in \textit{exploded manifolds} (certain log-schemes have the structure of exploded manifolds \cite{Parker}). Alternatively, from the perspective of the LG model $(X \times \CC, z f)$, the critical locus $f\inv(0) \times \CC$ comes with a \textit{$(-1)$-shifted symplectic structure}: we conjecture that this extra structure is also sufficient to determine the wrapped Fukaya category, which could be constructed intrinsically using a formalism such as Joyce's theory of $d$-critical loci \cite{Joyce}.

Mirror symmetry results where the symplectic side is an LG model of the form $(X \times \CC, z f)$, can be recast via Theorem \ref{thm:main_thm} into mirror symmetry statements about the singular variety $f\inv(0)$. For instance, two papers by Nadler  \cite{NadlerI, NadlerII} study a microlocal version of the A-model of $(\CC^{n+1}, z_1 \cdots z_{n+1})$. In this setting, our periodicity theorem admits another proof, which we shall sketch in \S \ref{sec:pants}, as well as a simple proof of Nadler's result in our terminology by applying work of Abouzaid-Auroux \cite{AA}. We illustrate how our definition allows this to be generalized to the case of complements of hypersurfaces in toric varieties, using ideas from the forthcoming sequel to \cite{AA}. 

This is a special case of a more general mirror symmetry statement, which can be understood in the context of the Gross-Siebert program (for notation, see \S \ref{sec:gs}). Suppose $B$ is an integral affine manifold, together with a polyhedral decomposition $\mathscr{P}$ and choice of a strictly convex integral piecewise-linear multivalued function $\varphi$, and let $(\check{B}, \check{\mathscr{P}}, \check{\varphi})$ be the discrete Legendre dual data. Let  
$\scr{X}_{\varphi} \to D$ and $\check{\scr{X}}_{\check{\varphi}} \to D$ be the toric degenerations of $X(B) = TB/T^{\ZZ}B$ and $X(\check{B}) = T \check{B}/T^{\ZZ}\check{B}$ respectively that are constructed by Gross and Siebert. A \textit{large volume limit} is the complement $X(B) \setminus s\inv(0)$ where $s$ is a section of the ample line bundle $\scr{L}_{\varphi}$ associated to $\varphi$, while a \textit{large complex structure limit} is the central fiber $\scr{X}_{\varphi}^0$ of a toric degeneration.

\begin{restatable}{theorem}{newtheorem}\label{thm:new_theorem}
When $B = M_{\RR}/\Gamma$ is an integral affine torus with a polyhedral decomposition $\mathscr{P}$ and choice of a strictly convex piecewise-linear multivalued function $\varphi$ with integral slopes, the large complex structure limit $\check{\scr{X}}_{\check{\varphi}}^0$ is homologically mirror to the large volume limit $X(B) \setminus s\inv(0)$, that is, there is a quasiequivalence of categories:
\begin{equation*}
    D^{\pi} \mathscr{F}(\check{\scr{X}}_{\check{\varphi}}^0) \simeq D^b \mathrm{Coh}(X(B)\setminus s\inv(0))
\end{equation*}
\end{restatable}

One corollary of this theorem is a mirror symmetry statement between the \textit{Fukaya category} of an elliptic curve with $n$ nodes, and the derived category of coherent sheaves of an elliptic curve with $n$ punctures. It is our understanding that these instances of homological mirror symmetry were not previously known. More generally, we have

\begin{restatable}{theorem}{lcsl}\label{thm:lcsl}
Given a homological mirror symmetry equivalence between $X(B)$ and $X(\check{B})$ that identifies the (clockwise) monodromy functor $\mu$ around the large complex structure limit of $X(\check{B})$ with the tensor product by the inverse of a line bundle $\scr{L}\inv$ on $X(B)$ whose first Chern class $c_1(\scr{L})$ is the class of the K\"ahler form on $X(B)$; then the large complex structure limit $X_0(\check{B})$ of $X(\check{B})$ is homologically mirror to a large volume limit of $X(B)$:
\begin{equation*}
    D^{\pi} \mathscr{F}(X_0(\check{B})) \simeq D^b \mathrm{Coh}(X(B)\setminus s\inv(0))
\end{equation*}
where $s$ is a section of $\scr{L}$.
\end{restatable}

A homological mirror symmetry equivalence of the form required to apply Theorem \ref{thm:lcsl} is expected to follow from combining Abouzaid's family Floer theory with ideas from the Gross-Siebert program (see \cite{theta}). Further explanation is provided in \S \ref{sec:gs}: these arguments are necessarily somewhat heuristic, neglecting points such as local systems, brane structures, and so forth. Theorem \ref{thm:lcsl} should be taken as an illustration that Definition \ref{defn:auroux} should yield expected mirror symmetry equivalences. 

\begin{remark}
Many strategies for proving homological mirror symmetry work in the opposite direction to Theorem \ref{thm:lcsl}: start by proving an equivalence for the large complex structure limit (LCSL, on the \textit{B-side}) and the large volume limit (LVL, on the \textit{A-side}), and then deduce HMS for the nearby fibers using a deformation theory argument. While understanding the Fukaya category of the LVL is easier than understanding that of the nearby fibers, the LCSL is a \textit{singular} variety and so the Fukaya category is not defined conventionally but only by reference to the nearby fibers. This provides some justification for the difference in approach; of course, one can combine both approaches to gain an understanding of HMS at all points of the moduli space.
\end{remark}

\subsection{Outline of Proof of Main Theorem}

We will first give a proof of an upgraded form of a theorem of Abouzaid-Auroux-Katzarkov \cite[Corollary 7.8]{AAK} in our setup, which may be of independent interest:

\begin{restatable}{theorem}{AAK}(Abouzaid-Auroux-Katzarkov Equivalence)\label{thm:aak}
Suppose that $X$ is a smooth affine variety with an embedding $X \to \CC^N$ inducing a Stein structure on $X$, and suppose $f:X \to \CC$ is the restriction of a polynomial function on $\CC^N$. If $f$ has a single critical fiber $f\inv(0)$ then for sufficiently small $t \neq 0$, we have a quasiequivalence of $A_{\infty}$-categories:
\begin{equation*}
    T:\scr{W}(f\inv(t)) \to \scr{W}(X \times \CC, z(f-t))
\end{equation*}
given by taking thimbles over admissible Lagrangians in the singular locus $f\inv(t)$.
\end{restatable}

In \cite{AAK} this was expected to be an equivalence of categories, assuming a generation result of Abouzaid-Ganatra \cite{AG} which should follow from Proposition \ref{prop:generation}. For a similar result, see \cite[Lemma A.26]{abouzaidSmith}. This theorem could be considered as an open-string analog of the LG-CY correspondence.

There is also a relative version of this statement, for fiberwise stopped Fukaya-Seidel categories, which mirrors a result of Orlov:
\begin{restatable}{theorem}{relAAK}\label{thm:rel_AAK}
Suppose that $X$ is a smooth affine variety with an embedding $X \to \CC^N$ inducing a Stein structure on $X$, and suppose $f, g:X \to \CC$ are restrictions of polynomial functions on $\CC^N$. If $f$ has single critical fiber $f\inv(0)$ then for sufficiently small $t \neq 0$ we have a fully faithful functor
\begin{equation*}
    \scr{W}(f\inv(t), g) \to \scr{W}(X \times \CC, z(f-t), g)
.\end{equation*}
\end{restatable}
Here the second category is fiberwise stopped with respect to $g$, as explained in Definition \ref{defn:relative}.

This should also imply a relative version of Theorem \ref{thm:main_thm}, which would then imply Conjecture \ref{conj:cis}:

\begin{conjecture}\label{conj:rel_thm}
Suppose $f:X \to \CC$ is a holomorphic function on a Stein manifold $X$ having a single critical fiber $f\inv(0)$; suppose $g:X \to \CC$ is another holomorphic function. Then for $\delta, |t|>0$ sufficiently small there is a quasiequivalence of $A_{\infty}$-categories
\begin{equation*}
    D^{\pi}\scr{W}(f\inv(t),g)[s\inv] \to D^{\pi}\scr{W}(X \times \CC, z f + \delta g)
.\end{equation*}
\end{conjecture}

We shall then show in \S \ref{sec:stein} that passing from $(X \times \CC, z(f-t))$ to $(X \times \CC, zf)$ can be rephrased as a stop-removal by carefully analyzing the Liouville geometry of the general fiber as $t$ changes. Hence, by the stop removal theorem of \cite{GPS2,Sylvan}, the category $\scr{W}(X \times \CC, zf)$ may be obtained as a quotient of the category $\scr{W}(X \times \CC, z(f-t))$ by a full subcategory $\scr{D}$ of linking disks. Lastly, we show in \S \ref{sec:proofs} that under the functor $T$ of Theorem \ref{thm:aak}, the essential image of the cones of the natural transformation $\mu \to \mathrm{id}$ on $\scr{W}(f\inv(t))$ split-generates the same full subcategory as $\scr{D}$, using a K\"unneth-type argument. Our Theorem \ref{thm:main_thm} then follows.

\subsection*{Acknowledgements}

First of all I would like to thank Denis Auroux for suggesting this problem as well as for his patience and important contributions throughout the process; I would also like to thank Mohammed Abouzaid, David Favero, Daniel \'Alvarez-Gavela, Sheel Ganatra, Paul Hacking, David Nadler, John Pardon, Vivek Shende, and Zack Sylvan for helpful conversations or correspondence, as well as Andrew Hanlon and Jeff Hicks for explaining their thesis work to me, and Paul Seidel for pointing out some inaccuracies in the published version. I would also like to thank the anonymous referee for many helpful comments and suggestions. This work was partially supported by the Rutherford Foundation of the Royal Society of New Zealand, NSF grant DMS-1937869, and by Simons Foundation grant \#385573.

This article differs from the version published in \textit{Adv. Math.} by corrections to \S 3, \S 5.3, new figures, updated references to papers that have now appeared, to agree with the version in the author's 2023 PhD thesis.

\clearpage

\section{Definitions and Conventions}\label{sec:defns}

We begin with a discussion of how to define the Fukaya-Seidel category of a Landau-Ginzburg model associated to a \textit{single} critical value. Suppose that $X$ is a smooth affine variety with an embedding $i: X \to \CC^N$; then $X$ becomes a Stein manifold with the Stein function $\phi: X \to \RR$ given by the restriction of $\phi(z) = |z|^2$ on $\CC^N$. Suppose that $f: \CC^N \to \CC$ is a polynomial with $0$ as its only critical value; by abuse of notation we will denote the restriction to $X$ also by $f$. Suppose for now that $\CC$ carries the standard Stein structure. We would like to turn the pair $(X,f)$ into a Landau-Ginzburg model, in particular, talk about its Fukaya-Seidel category. 

Firstly, the function $|f|^2$ defines a real polynomial on $\RR^{2N}$, and $X \subset \RR^{2N}$ is a real affine algebraic variety. It is a well-known result from real algebraic geometry that the set of points in $\RR$ for which the Malgrange condition for $|f|^2:X \to \RR$ fails is finite (for instance see \cite[Remark 3]{Malgrange} and take the intersection with the algebraic variety $X$). Recall that this Malgrange condition at $\lambda \in \RR$ says that there exists $R, \epsilon, \eta>0$ so that if $|z|>R$ and $||f|^2(z) - \lambda| \leq \epsilon$ then
\begin{equation*}
    |z| |\grad_X |f|^2| > \eta
.\end{equation*}
Without loss of generality, $|f|^2(0) = 0$, and since we have that $|\grad_X |f|^2| >0$ for $|f|^2 > 0$, we see that for any $\lambda \in \RR_{>0}$ for which the Malgrange condition holds, there exists a $C>0$ (depending on $\lambda$) so that on $|f|^2 = \lambda$ we have
\begin{equation*}
    |f|^2 < C |z| |\grad_X |f|^2|
.\end{equation*}
Now take $\delta>0$ strictly smaller than all points in $\RR_{>0}$ where the Malgrange condition fails for $|f|^2$ on $X$. Then there exists some constant $C$ so that for all $0<|f|^2 < \delta$ we have:
\begin{equation*}
    |f|^2 < C |z| |\grad_X |f|^2|
.\end{equation*}
Choose $m \in \NN$ sufficiently large so that $m>4C$ and for any constant $D>0$ define a Stein function via
\begin{equation*}
    \psi(z) = \phi(z) + D\phi(z)|f|^{2m}
\end{equation*}
which can be induced using the algebraic embedding $\tilde{i}: X \to \CC^{2N}$ given by $z \mapsto (i(z), \sqrt{D} f^m(z) i(z))$. This has a regular homotopy to the original Stein structure given by the family with $t\in[0,1]$
\begin{equation*}
    \psi_t(z) = \phi(z) + D\br{\frac{t + \phi(z)|f|^{2m}}{1 + t \phi(z)|f|^{2m}}}
.\end{equation*}

\begin{proposition}\label{prop:outwards}
The Liouville vector-field of $\psi$ is outward pointing along $|f|^2 = \delta$ for $D>0$ sufficiently large.
\end{proposition}
\begin{proof}
We study the inner product
\begin{equation*}
    \langle \grad_X |f|^2, \grad_X \psi \rangle  = (1 + D|f|^{2m}) \langle \grad_X |f|^2, \grad_X \phi \rangle + m D |f|^{2m-2} \phi |\grad_X |f|^2|^2
.\end{equation*}
By the Cauchy-Schwarz inequality
\begin{equation*}
    (1 + D|f|^{2m}) |\langle \grad_X |f|^2, \grad_X \phi \rangle| \leq 2 (1 + D |f|^{2m}) |\grad_X |f|^2| |\phi|^{1/2}
\end{equation*}
and thus we will have
\begin{equation*}
    \langle \grad_X |f|^2, \grad_X \psi \rangle > 0
\end{equation*}
so long as
\begin{equation*}
    2(1 + D |f|^{2m}) <  m D |f|^{2m-2}|\phi|^{1/2} |\grad_X |f|^2|
\end{equation*}
since 
\begin{equation*}
    |\grad_X \phi| \leq |\grad |z|^2| = 2 |z|
.\end{equation*}
Rewriting this inequality gives
\begin{equation*}
    2\br{ \frac{1}{D |f|^{2m-2}} + |f|^2 } < m |\phi|^{1/2}|\grad_X |f|^2|
.\end{equation*}
If we take $D>1/\delta^{m}$ then
\begin{equation*}
    \frac{1}{D |f|^{2m-2}} < \delta
\end{equation*}
and hence
\begin{equation*}
     4 \delta <  m |z| |\grad_X |f|^2|
\end{equation*}
follows by our construction.
\end{proof}

Now define a stopped Liouville domain as follows. First, take the subset $X_0 = \set{|f|^2 \leq \delta} \cap \set{\psi \leq R} \subset X$, which by Proposition \ref{prop:outwards}, for $R$ sufficiently large and $\delta$ sufficiently small, gives a Liouville domain $\tilde{X}_0$ after rounding the corners. Then for a sufficiently small rounding, the hypersurface $F = f\inv(-\delta) \cap \set{\psi \leq R-\epsilon} $ sits inside $\set{|f|^2 = \delta} \cap \set{\psi \leq R-\epsilon}$, which is part of the contact boundary of $\tilde{X}_0$. The pair $(\tilde{X}_0,F)$ a sutured Liouville domain, and so by \cite[Lemma 2.32]{GPS} there is a (homotopically) unique Liouville sector associated to it which we will denote by $(X,f)$; moreover, this is independent of $R$ sufficiently large and $\delta$ sufficiently small. We will write $\scr{W}(X,f)$ for its partially wrapped Fukaya category as defined in \cite{GPS2}: this is the (fiberwise wrapped) \textbf{Fukaya-Seidel category} of the Landau-Ginzburg model $f:X \to \CC$ (associated to the single critical value $0$). 

\begin{remark}\label{remark:conventions}
We shall henceforth assume that this procedure has taken place and that the Stein function $\phi$ on $X$ is already given by $\psi$. We can always equivalently consider $(X,f)$ as a Liouville manifold with a stop $f\inv(-\infty)$ sitting inside the contact boundary at infinity. In this case, though we may informally discuss our Landau-Ginzburg models as living over all of $\CC$, their corresponding Liouville sectors are constructed as above and lie only over a small disk $\set{|z|^2 \leq \delta} \subset \CC$.
\end{remark}


Note that above all of the non-zero points where the Malgrange condition failed to hold are made to lie outside the Liouville domain. This has several useful consequences.

\begin{lemma}\label{lem:transport}
Symplectic parallel transport gives exact symplectomorphisms between smooth fibers of $f$ over $0<|f| < \epsilon$ for some $\epsilon>0$.
\end{lemma}
\begin{proof}
We follow the argument of \cite[\S 2]{FSS}. By a simple calculation (cf. \cite[p.213]{SeidelBook}) it follows that the symplectic parallel transport gives exact symplectomorphisms whenever it is defined. The parallel transport vector field is always a unit complex multiple of
\begin{equation*}
    \frac{\grad_X \re(f)}{|\grad_X \re(f)|^2}
\end{equation*}
since $\nabla_X \mathrm{Im}(f) = J \nabla_X \mathrm{Re}(f)$. Because $\phi:X \to \RR$ is a proper exhausting function, it suffices to show that the image under $\phi$ of the flow lines $\gamma(t)$ of the parallel transport vector field do not escape to infinity in finite time. Therefore consider the derivative
\begin{equation*}
    \left| \frac{\grad_X \Re(f)}{|\grad_X \Re(f)|^2} \phi \right| =  \frac{\left| \langle \grad_X \Re(f), \grad_X \phi \rangle \right|}{|\grad_X \Re(f)|^2} \leq \frac{|\grad_X \phi|}{|\grad_X \Re(f)|}
.\end{equation*}
Since there are only finitely many points in $\CC$ where the Malgrange condition fails for the \textit{complex} polynomial $f:X \to \CC$, we have a constant $C>0$ and some $\epsilon>0$ so that 
\begin{equation*}
    |f|< C |\phi|^{1/2} |\dd_X f|
\end{equation*}
for $0<|f|<\epsilon$, where $|\dd_X f| = |\grad_X \Re(f)|$. Observe also that on $\CC^N$ we have
\begin{equation*}
    |\grad \phi| = |\grad |z|^2| = 2 |\phi|^{1/2}
\end{equation*}
and hence
\begin{equation*}
    |\grad_X \phi| \leq 2 |\phi|^{1/2}
.\end{equation*}
Therefore we have that on $0<|f|<\epsilon$, for some constant $C$,
\begin{equation*}
    \phi\dash(t) \leq \frac{C\phi(t)}{|f(t)|}
\end{equation*}
where $\phi(t) = \phi(\gamma(t))$. If $\gamma(t)$ avoids the singular fiber of $f$, there is some constant $\beta>0$ so that
\begin{equation*}
    \frac{1}{|f(t)|} < \beta
.\end{equation*}
Hence by Gr\"onwall's inequality
\begin{equation*}
    \phi(t) \leq \phi(0)\e{C \beta t}
\end{equation*}
which completes the proof.
\end{proof}

\begin{remark}
By shrinking $\delta>0$ further in Proposition \ref{prop:outwards}, we will assume that $\epsilon$ from Lemma \ref{lem:transport} has $\delta<\epsilon$.
\end{remark}

Often we shall want to have the freedom to use a different Stein function to compute Fukaya categories. We now record here a standard lemma we shall use throughout, which however only appears implicitly in the literature.

\begin{definition}\label{def:simple}
Suppose $X$ is a Liouville manifold with Liouville form $\lambda$; a smooth family $\lambda_t$, $t \in [0,1]$ of Liouville forms for $X$ with $\lambda_0=\lambda$ such that the  union of the skeleta of the Liouville structures $\lambda_t$ stays within a compact set, is called a \textbf{simple Liouville homotopy}.
\end{definition}

For instance, this condition is satisfied by a family of Weinstein functions whose critical points remain in a compact set.

\begin{lemma}\label{lem:deformation}
Given a Liouville manifold $X$ and a simple Liouville homotopy $\lambda_t$, all of the Liouville manifolds $(X, \lambda_t)$ are exact symplectomorphic and the wrapped Fukaya categories $\scr{W}(X,\lambda_t)$ are all quasiequivalent.
\end{lemma}
\begin{proof}
By \cite[Proposition 11.8]{CE} for every simple Liouville homotopy there is a family of exact symplectomorphisms $\phi_t:X \to X$ with $\phi_{t}^{\ast}\lambda_t = \lambda - \dd f$ with $f$ \textit{compactly supported}, such that $\phi_0$ is the identity. This gives rise to a trivial inclusion of Liouville sectors and so by \cite[Lemma 2.6]{GPS2} this deformation yields a quasiequivalence of wrapped Fukaya categories between $\scr{W}(X,\lambda_0)$ and $\scr{W}(X,\lambda_1)$.
\end{proof}

We are hence free to work with deformed Liouville structures in our proofs when considering only the Fukaya category up to quasi-equivalence (by \cite[Corollary 1.14]{SeidelBook}). There is a similar version of Lemma \ref{lem:deformation} also in the stopped case. A particularly useful application of this is as follows.

\begin{proposition}\label{prop:product}
Suppose $U \subset \CC$ is a simply-connected open set inside $\set{0<|z|<\delta}$; then there is a simple Liouville deformation on $X$ (supported in a neighbourhood of $f\inv(U)$) that makes $f\inv(U)$ exact symplectomorphic to the product $f\inv(\delta) \times U$ with the product 1-form, preserving the 1-forms of the fibers.
\end{proposition}

The proof corresponds to the discussion around \cite[Lemma 15.3]{SeidelBook}. Summarized, by Lemma \ref{lem:transport} we can trivialize $f$ over $U$ using symplectic parallel transport, so that smoothly \begin{equation*}
    f\inv(U) \cong f\inv(\delta) \times U
.\end{equation*}
Since the symplectic parallel transport induces exact symplectomorphisms of the fibers, up to an exact form, the symplectic form pulled back to $f\inv(\delta) \times U$ is given by
\begin{equation*}
   \lambda =  \lambda_{F} + \kappa
\end{equation*}
where $\lambda_F = \lambda|_{f\inv(\delta)}$ and $\kappa$ is the symplectic connection form. There is then a deformation
\begin{equation*}
    \lambda_s = \lambda_F + s \kappa + c(1-s) \pi^{\ast} \lambda_U
\end{equation*}
for $s \in [0,1]$ and $c>0$ sufficiently large. 
As in \cite[Lemma 15.3]{SeidelBook} this can be supported in a neighbourhood of $f\inv(U)$ and is a simple deformation since it preserves the fibers.

\subsection{Cap and Cup Functors}\label{sec:cup}

Now we wish to define some functors introduced in \cite{AS} in the language of \cite{GPS2}. Firstly, if $F \subset \partial^{\infty}X$ is a Liouville hypersurface, then there is the \textit{Orlov functor}
\begin{equation*}
    \scr{W}(F) \to \scr{W}(X,F)
.\end{equation*}
given by taking small counterclockwise linking disks of the stop $F$ \cite{GPS2, Sylvan2}. Henceforth we shall assume that $F$ is the fiber of a Landau-Ginzburg model $f:X \to \CC$ as explained above in \S \ref{sec:defns}. In this case $F = f\inv(-\delta)$ we call the functor $\scr{W}(F) \to \scr{W}(X,F)$ the \textit{cup functor} $\cup$: see Figure \ref{fig:cup}. 

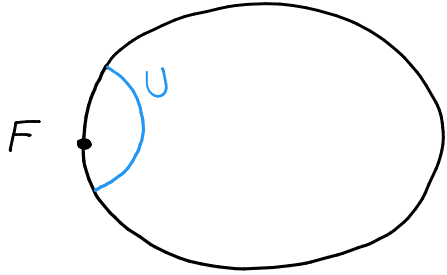
\begin{figure}
\centering
\begin{tikzpicture}[scale=1]
\filldraw[fill=gray!10] (0,0) circle (2);
\draw[fill=black] (-2,0) circle (0.05);
\draw[thick,blue] (-150:2) arc[start angle=-90, end angle=90,radius=1];
\node[anchor=north west] at (-1,1) {$\cup$};
\node[anchor=north east] at (-2,0) {$F$};
\end{tikzpicture}
    \caption{The cup functor.\label{fig:cup}}
\end{figure}

The cup functor has a formal adjoint given by the pullback on left Yoneda modules $\mathrm{Mod}\scr{W}(X,F)\to \mathrm{Mod}\scr{W}(F)$, which we call the \textit{cap functor} $\cap$. We then have counit and unit morphisms $\epsilon: \cup \cap \to \mathrm{id}$ and $\eta: \mathrm{id} \to \cap \cup$ respectively, which we may complete to exact triangles of bimodules. These exact triangles in fact have a geometric characterization in terms of earlier work of Seidel:

\begin{theorem*}(Abouzaid-Ganatra, \cite{AG})
There are exact triangles of bimodules
\begin{center}
\begin{tikzcd}
\mathrm{\cup \cap} \ar{rr}{\epsilon} & {} & \mathrm{id} \ar{dl}\\
{} & \sigma \ar{ul}{+1} & {}
\end{tikzcd}
\end{center}
on $D^{\pi} \scr{W}(X,F)$ and
\begin{center}
\begin{tikzcd}
\cap \cup \ar[swap]{dr}{+1} & {} & \mathrm{id} \ar[swap]{ll}{\eta}\\
{} & \mu \ar{ur}{s} & {}
\end{tikzcd}
\end{center}
on $D^{\pi}\scr{W}(F)$, where $\mu: \scr{W}(F) \to \scr{W}(F)$ is the clockwise monodromy acting on the fiber and $\sigma: \scr{W}(X,F) \to \scr{W}(X,F)$ is the clockwise total twist acting on Lagrangians in the total space. Moreover, the natural transformation $s$ may be identified with Seidel's natural transformation, first introduced in \cite{SHasHH} (see also \cite{Subalgebras, LFII}).
\end{theorem*}

Conventions for these two triangles differ: see \cite[Theorem 1.3]{Sylvan2} for the identifications of the twist/cotwist with the monodromy functors and \cite[Appendix]{abouzaidSmith} for a proof of one triangle. Our conventions are chosen to be compatible with the counterclockwise wrapping of \cite{GPS} so that there is a degree-$0$ natural transformation $\mathrm{id} \to \sigma$: this forces $\cup, \sigma, \mu$ to be clockwise and $\cup$ the left adjoint. However, this means that in the case of a model Lefschetz fibration, $\mu$ is the negative Dehn twist. 

\textit{None of these results are logically necessary for the proof of Theorem \ref{thm:main_thm} of this paper}, and for the proofs following the reader may take $\mu, \sigma, s$ to be defined via the purely algebraic definition given above. Then we may make the definition:

\auroux*

The condition that $|t| \neq 0$ be sufficiently small is made to exclude the possibility that $t$ is a point where the Malgrange condition fails; in our definitions earlier in \S \ref{sec:defns} we define our Landau-Ginzburg models so as to exclude such points. 

Note that this definition is equivalent to taking the quotient of the category $\scr{W}(f\inv(t))$ by the full subcategory of the cones of the natural transformation, that is, the quotient by the essential image of the composition $\cap \cup$. The following lemma is often useful for computations:

\begin{lemma}\label{lemma:split}
The essential image of the composition $\cap \cup$ has the same split-closure as the essential image of $\cap$.
\end{lemma}
\begin{proof}
One inclusion is clear. To show the reverse inclusion, consider the two exact triangles relating the composition of the $\cap$ and $\cup$ functors; importantly, note that these two triangles have arrows in reverse directions with the conventions taken above, where $\epsilon$ and $\eta$ are the counit and unit of the $\cup-\cap$ adjunction respectively. We shall apply $\cap$ to the first triangle and precompose the second triangle with $\cap$. Exactness and the identity $\mu \cap = \cap \sigma$ \cite{AG} yields two triangles: 

\begin{center}
\begin{tikzcd}
\cap \cup \cap \ar{rr}{\cap \epsilon} & {} & \cap \ar[swap]{dl}\\
{} & \mu \cap \ar{ul}{+1} & {}
\end{tikzcd}
\end{center}

\begin{center}
\begin{tikzcd}
\cap \cup \cap \ar[swap]{dr}{+1} & {} & \cap \ar[swap]{ll}{\eta \cap}\\
{} & \mu \cap \ar{ur} & {}
\end{tikzcd}
\end{center}

The unit-counit identity $\cap \epsilon \circ \eta \cap = \mathrm{id}_{\cap}$ for an adjunction implies that the map $\mu \cap \to \cap$ in the second triangle is zero and hence that this triangle splits. Thus $\cap \cup \cap L$ always has $\cap L$ as a direct summand.
\end{proof}

\begin{corollary}\label{cor:alt_defn}
The category $D^{\pi}\scr{W}(f\inv(0))$ is quasiequivalent to the quotient of the category $D^{\pi}\scr{W}(f\inv(t))$ by the essential image of $\cap$.
\end{corollary}
\begin{proof}
By Lemma \ref{lemma:split}, the essential images of the two functors $\cap \cup$ and $\cap$ have the same split-closure. Thus the resulting quotients will be quasiequivalent.
\end{proof}

\begin{remark}
To see that it is necessary to take the split-closure of the twisted complexes, one can consider the instance where the monodromy around the singular fiber is given by the composition of two Dehn twists around the same vanishing cycle. 
\end{remark}

The reader wishing to avoid \cite{AG,Sylvan2} may instead take Corollary \ref{cor:alt_defn} as a definition for the purposes of the proof of Theorem \ref{thm:main_thm}.

\label{sec:cobs} Definition \ref{defn:auroux} also has a particularly natural relation to Lagrangian cobordisms inside symplectic fibrations (cf. \cite{Cob}). In place of studying Lagrangian cobordisms inside the product $f\inv(t)\times \CC$, we could alternatively consider Lagrangian cobordisms inside $X$ (or several concatenated copies of $X$) with ends projecting via $f$ to rays parallel to the positive or negative real axes: see Figure \ref{fig:cob}. After applying a Hamiltonian isotopy, every such cobordism may instead be considered as having only positive ends; these nullcobordisms thus represent all the equivalence relations imposed on the group of Lagrangian cobordisms. But these relations say that every complex in the image of $\cap$ must be equivalent to zero, which by Corollary \ref{cor:alt_defn} is an equivalent description of the Fukaya category $\scr{W}(f\inv(0))$. We therefore conjecture that this construction describes the Grothendieck group $K_0(D^{\pi}\scr{W}(f\inv(0)))$.

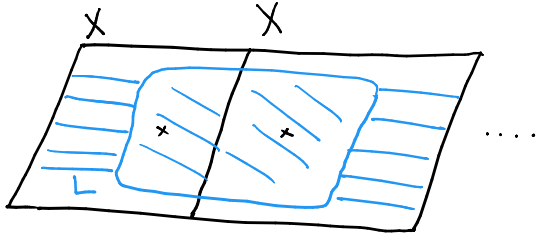
\begin{figure}
\centering
\begin{tikzpicture}[scale=0.5]
\path[draw,fill=blue!20,use Hobby shortcut,closed=true] (-4,-4) .. (3,-3) .. (4.25,0) ..(3,3) .. (3,4) .. (0,5.5) .. (-4,3) .. (-3.75,1) .. (-3.75,-1) .. (-4,-4);
\draw[thick,blue] (4.25,0) to (6,0); 
\draw[thick,blue] (3,3) to (6,3);
\draw[thick,blue] (3,-3) to (6,-3);
\draw[thick,blue] (-3.75,-1) to (-6,-1); 
\draw[thick,blue] (-3.75,1) to (-6,1); 
\draw[thick] (-6,-6) to (6,-6) to (6,6) to (-6,6) to (-6,-6);
\draw[thick,dashed] (0,-6) to (0,6);
\draw (-2.5,0) node {$\huge{\times}$};
\draw (2.5,0) node {$\huge{\times}$};
\draw (-2,1) node {$\huge{\times}$};
\draw (3,1) node {$\huge{\times}$};
\draw (-2,-1) node {$\huge{\times}$};
\draw (3,-1) node {$\huge{\times}$};
\draw (5,-5) node {$X$};
\draw (-5,-5) node {$X$};
\draw[blue] (-4,4.5) node {$L$};
\end{tikzpicture}
    \caption{A Lagrangian cobordism $L$ inside a concatenation of copies of $X$. \label{fig:cob}}
\end{figure}

\section{Complements of Fibers}\label{sec:stein}

We begin with some simple observations about the geometry of the Landau-Ginzburg models $(X \times \CC, z(f-t))$ for different values of $t$. For the purposes of this section, let $f_t(z) = f(z) - t$.

Firstly, observe that $z(f-t)$ has only a single critical fiber, occurring where $z(f-t) = 0$. This fiber is given set-theoretically by $(f\inv(t) \times \CC) \cup (X \times \set{0})$; the critical locus is given exactly by $f\inv(t)$. Observe that for $t \neq 0$, this critical locus is therefore smooth, and we have an explicit local Morse-Bott model, described in further detail below.

For $\epsilon \neq 0$, the smooth fiber of $z(f-t)$ over $\epsilon$ is given by $\set{(x,z): f(x) \neq t\; \text{and} \; z = \epsilon/f_t}$, which can be identified set-theoretically with the complement $X\setminus f\inv(t)$. The corresponding Liouville structure on this complement comes from the restriction of the Stein function to the fiber. Using the Stein function $|z|^{2n}$ on the factor $\CC$ this is hence given by
\begin{equation*}
    \psi(x) = \phi(x) + \frac{\epsilon^{2n}}{|f_t(x)|^{2n}}
.\end{equation*}
This Stein function may be obtained by taking the embedding of $X \setminus f\inv(t)$ into $\CC^{N+1}$ given by $j(z) = \br{i(z), \frac{\epsilon}{f_t(z)}}$. We shall instead use a deformation-equivalent Stein structure on the complement, given by
\begin{equation*}
    \psi(x) = \phi(x) + \frac{C \phi(x)}{|f_t(x)|^{2n}}
\end{equation*}
for suitable constants $C>0$ and $n \in \NN$. This Stein function is given by the embedding of $X\setminus f\inv(t)$ into $\CC^{2N}$ given by $z \mapsto (i(z), \sqrt{C}i(z)/f_t(z)^n)$ and is Stein deformation-equivalent to the previous under the deformation for $s \in [0,1]$ given by
\begin{equation*}
    \psi_s(x) = \phi(x) + \frac{s + \phi(x)}{1 + s \phi(x)}\frac{C}{|f_t(x)|^{2n}}
.\end{equation*}

To understand how the LG models of $(X \times \CC, z(f-t))$ for $t \neq 0$ and $t=0$ differ, we need to study how the Liouville structure of the general fiber changes. The following elementary example illustrates the procedure.\\

\begin{example}\label{example}
We recall how to build standard Weinstein structures on Lefschetz fibrations. Consider $f:\CC^{2} \to \CC$ the standard Lefschetz fibration, and equip $\CC^2$ with the standard Stein structure. Consider the skeleton of the complement of the fiber $f\inv(t)$. For $t=0$, we see that $\CC^2 \setminus f\inv(0)  = (\CC^{\ast})^2$ with the standard Liouville structure, so its skeleton corresponds to the zero-section $T^2 \subset T^{\ast} T^2 \cong (\CC^{\ast})^2$.

Observe that the Weinstein function coming from the Stein structure on $\CC^2$ has a single critical point at $0$, of index $2$. When we take $t \neq 0$, the Weinstein function on $X \setminus f\inv(t)$  now includes this critical point. Therefore $\CC^2 \setminus f\inv(t)$ is obtained from $\CC^2 \setminus f\inv(0)$ by attaching a critical Weinstein handle to $\CC^2 \setminus f\inv(0)$. The skeleton of $\CC^2 \setminus f\inv(t)$ now includes the core of this handle, attached to the original $T^2$, as in Figure \ref{fig:example}. \QEDB
\end{example}

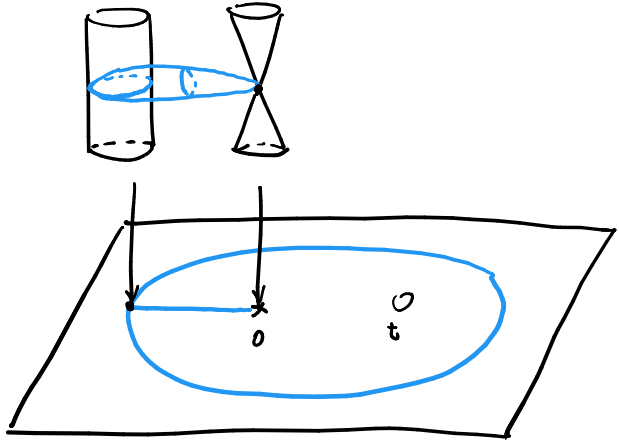
\begin{figure}
\centering
\begin{tikzpicture}[scale=0.75]
\draw[thick] (-5,-4) to (3,-4) to (5,0) to (-3,0) to (-5,-4);
\filldraw[blue,fill=blue!20,thick] (-4,2.5) to[out=70,in=150] (0,2.5) to[out=210,in=-70] (-4,2.5);
\draw (0,-2) node {$\huge{\times}$};
\draw[dashed] (0,-2) -- (0,1);
\draw[dashed] (-3,-2) -- (-3,1);
\draw[fill=white] (2,-2) circle (0.1);
\draw[fill=black] (-3,-2) circle (0.1);
\draw[blue,thick] (0,-2) ellipse (3 and 1.5);
\draw[blue,thick] (-3,-2) to (0,-2);
\draw (-3,4) ellipse (1 and 0.5);
\draw (-3,1) ellipse (1 and 0.5);
\draw[blue] (-3,2.5) ellipse (1 and 0.5);
\draw (-4,4) to (-4,1);
\draw (-2,4) to (-2,1);
\draw (0,4) ellipse (1 and 0.5);
\draw (0,1) ellipse (1 and 0.5);
\draw (-1,4) to (1,1);
\draw (1,4) to (-1,1);
\draw[blue,dashed] (-1.5,2.5) ellipse (0.25 and 0.65);
\draw (0,-2) node[anchor=north east] {$0$};
\draw (2,-2) node[anchor=north west] {$t$};
\end{tikzpicture}
    \caption{Skeleton in Example \ref{example} \label{fig:example}(in blue).}
\end{figure}

To describe the general case we shall need a dual version of Proposition \ref{prop:outwards}.

\begin{restatable}{proposition}{inward}\label{prop:inward} For sufficiently large $C>0$ and $n \in \NN$, and $t$ sufficiently small the Liouville vector field of $\psi$ is inward-pointing on $\set{|f_t|^2 = \delta}$ for all $\delta>0$ sufficiently small.
\end{restatable}
\begin{proof}
We consider the inner product 
\begin{equation*}
    \langle \grad_X |f_t|^2, \grad_X \psi \rangle  = \br{1 + \frac{C}{|f_t|^{2n}}}\langle \grad_X |f_t|^2, \grad_X \phi \rangle - \frac{n C \phi}{|f_t|^{2n+2}} |\grad_X |f_t|^2|^2  
.\end{equation*}
By applying Cauchy-Schwarz, we see that we will have
\begin{equation*}
    \langle \grad_X |f_t|^2, \grad_X \psi \rangle  < 0
\end{equation*}
on $|f_t|\inv(\delta)$ whenever
\begin{equation*}
    2\br{\frac{|f_t|^{2n+2}}{C} + |f_t|^2} < n  |\phi|^{1/2} |\grad_X |f_t|^2|
.\end{equation*}
Again, by choosing $|t|, \delta>0$ sufficiently small we can ensure that the Malgrange condition is satisfied on $|f_t|^2=\delta$ for some sufficiently large $n \in \NN$ and this proposition will follow.
\end{proof}

\begin{remark}\label{remark:both}
Note that we can apply the same argument with the Stein function
\begin{equation*}
    \psi = \phi + C \phi |f_t|^{2m}
\end{equation*}
homotopic to the Stein function from Proposition \ref{prop:outwards} in place of $\phi$ so as to have both Proposition \ref{prop:outwards} and Proposition \ref{prop:inward} hold simultaneously (on concentric disks). The resulting inequality
\begin{equation*}
    C |f_t|^{2} + |f_t|^{2n+2}  + D |f_t|^{2m +2n+2} < \frac{\phi^{1/2}}{2}|\grad_X |f_t|^2|\br{ C n - D m |f_t|^{2m+2n}}
\end{equation*}
holds on $\epsilon \leq |f_t|^2 \leq \delta$ for $\delta > \epsilon>0$ sufficiently small and $C, n$ chosen sufficiently large, given $D,m$.
\end{remark}

We now apply Eliashberg's surgery theory for Weinstein manifolds to understand how the general fiber changes as we vary $t$:

\begin{proposition}\label{prop:surgery}
The Weinstein structure above on $X \setminus f\inv(t)$ is obtained from $X\setminus f\inv(0)$ by attaching a collection of Weinstein handles.
\end{proposition}
\begin{proof}
First we shall describe how to start with $X\setminus f\inv(0)$ and pass to $X\setminus f\inv(B_{\delta})$, both equipped with the Stein function:
\begin{equation*}
    \psi_0(x) = \phi(x) + \frac{C \phi(x)}{|f(x)|^{2n}}
.\end{equation*}
By Proposition \ref{prop:inward}, for $\delta>0$ sufficiently small and $n,C$ sufficiently large,  we can take $B_{\delta}$ a ball around $0$ so that the Liouville vector field for $\psi_0$ on $X$ is inward-pointing on $|f|\inv(\delta)$. Possibly shrinking $\delta>0$ further, there are no zeroes of the Liouville vector field in $f\inv(B_{\delta}\setminus \set{0})$, since the zeroes of $\phi$ on $X$ are isolated under the map $f:X \to \CC$. Hence we have a trivial Weinstein cobordism between $X\setminus f\inv(0)$ and $X\setminus f\inv(B_{\delta})$, and hence an isomorphism of Liouville manifolds \cite{CE}. 

Taking $|t|<\delta$, we now pass from $X\setminus f\inv(B_{\delta})$ to $X \setminus f\inv(t)$. First, there is a  deformation of Weinstein structures on $X\setminus f\inv(B_{\delta}(s))$ (all diffeomorphic for $|s| < \delta$) given by:
\begin{equation*}
    \psi_s(x) = \phi(x) + \frac{C \phi(x)}{|f_s(x)|^{2n}}
\end{equation*}
for $s \in [0,t]$. For $\epsilon >0$ sufficiently small, none of the critical points of $\psi_s$ outside $f\inv(B_{\delta})$ enter $f\inv(B_{\delta}(s))$ as $s$ goes from $0$ to $t$: hence this is a simple Liouville deformation and so $\psi_0$ and $\psi_t$ give isomorphic Liouville structures on the manifolds $X\setminus f\inv(B_{\delta})$ and $X\setminus f\inv(B_{\delta}(t))$ (Lemma \ref{lem:deformation}).

Now, since the Liouville vector field for $\psi_t$ is inward-pointing along $f\inv(\partial B_{\delta})$ note that $f\inv(B_{\delta})\setminus f\inv(t)$ gives a Weinstein cobordism from $X\setminus f\inv(B_{\delta})$ to $X \setminus f\inv(t)$. Hence by Eliashberg's surgery theory for Weinstein manifolds \cite{CE}, we obtain $X \setminus f\inv(t)$ by attaching Weinstein handles along the boundary of $X \setminus f\inv(B_{\delta})$: see Figure \ref{fig:handles}.
\end{proof}

\begin{figure}
\centering
\begin{tikzpicture}[scale=0.5]
\draw[thick] (0,0) circle (4);
\draw[thick] (-2,0.2) circle (0.25);
\draw[thick] (-5,-5) to (5,-5) to (5,5) to (-5,5) to (-5,-5);
\draw (-2,0.2) node {$\huge{\times}$};
\draw[fill=black] (2,0.2) circle (0.1);
\draw (2,0.2) node[anchor=west] {$t$};
\begin{scope}[shift={(-5.5,-5.5)},scale=2]
\begin{axis}[
    xmin = -2, xmax = 2,
    ymin = -2, ymax = 2,
    zmin = 0, zmax = 2,
    axis equal image,
    yticklabels={,,},
    xticklabels={,,},
    view = {0}{90},
    axis line style={draw=none},
]
    \addplot3[
        quiver = {
            u = {(-0.8-x)/100},
            v = {(-y)/100}, scale arrows=10
        }, 
        -stealth,
        domain = -5:5,
        domain y = -5:5,
        samples=30,
    ] {0};
\end{axis}
\end{scope}
\draw (-5,-6.5) node {$X\setminus f\inv(0)$};
\end{tikzpicture}\begin{tikzpicture}[scale=0.5]
\draw[thick] (0,0) circle (4);
\draw[thick] (2.2,0.2) circle (0.2);
\draw (2.5,0.2) node[anchor=west] {$t$};
\draw[thick] (-5,-5) to (5,-5) to (5,5) to (-5,5) to (-5,-5);
\draw (-2,0.2) node {$\huge{\times}$};
\draw[thick,red] (-2,0.2) to (-4,0.2);
\draw[thick,blue] (-2,0.2) to (2,0.2);
\begin{scope}[shift={(-5.5,-5.5)},scale=2]
\begin{axis}[
    xmin = -2, xmax = 2,
    ymin = -2, ymax = 2,
    zmin = 0, zmax = 2,
    axis equal image,
    yticklabels={,,},
    xticklabels={,,},
    view = {0}{90},
    axis line style={draw=none},
]
    \addplot3[
        quiver = {
            u = {(0.7-x)/100},
            v = {(-y)/100}, scale arrows=10
        }, 
        -stealth,
        domain = -5:5,
        domain y = -5:5,
        samples=30,
    ] {0};
\end{axis}
\end{scope}
\draw (-5,-6.5) node {$X\setminus f\inv(t)$};
\end{tikzpicture}
    \caption{The handle attachment procedure from Proposition \ref{prop:surgery}: cores of additional handles are in red and cocores in blue. Length of vectors not to scale.\label{fig:handles}}
\end{figure}
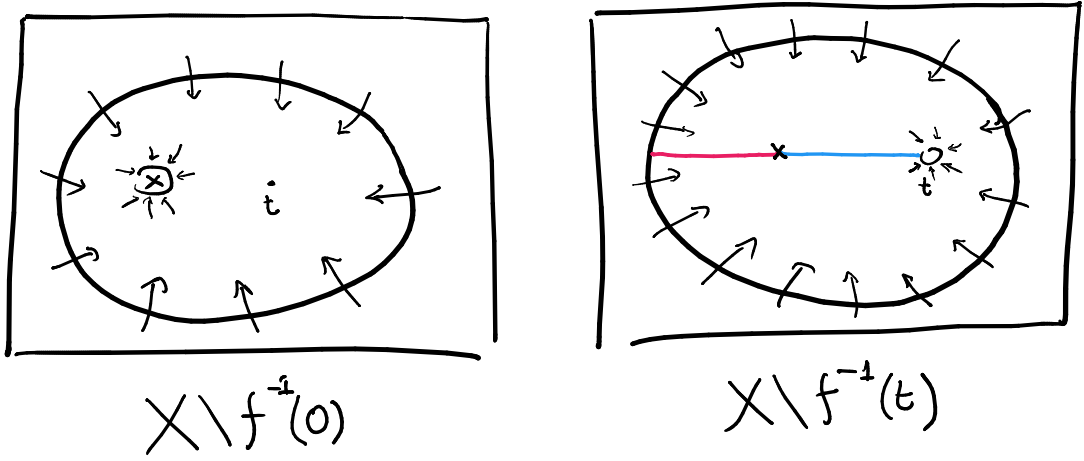

Any critical handles given by Proposition \ref{prop:surgery} will be referred to as \textbf{additional handles}. By applying the stop-removal theorem of Ganatra-Pardon-Shende \cite[Theorem 1.16]{GPS2}, we deduce:

\begin{proposition}\label{prop:stops}
With the Liouville structures constructed above, we have an equivalence of categories between $\scr{W}(X \times \CC, z f)$ and the quotient of $\scr{W}(X \times \CC, z(f-t))$ by the full subcategory $\scr{D}$ of linking disks of the stable manifolds of the additional handles in $X\setminus f\inv(t)$:
\begin{equation*}
    \scr{W}(X \times \CC, z f) \cong \scr{W}(X \times \CC, z(f-t))/\scr{D}
.\end{equation*}
\end{proposition}

\begin{proposition}\label{prop:generation}
For $f$ a function on a Stein manifold $X$ as above, the category $\scr{W}(X,f)$ is split-generated by the cocores of the additional handles.
\end{proposition}
In the course of the proof we shall explain the sense in which the cocores of the additional handles, which we will denote $\ell_i$, are objects of $\scr{W}(X,f)$.
\begin{proof}
We begin by studying the wrapped Fukaya category of $X\setminus f\inv(B_{\delta})$ with the Stein structure given by:
\begin{equation*}
    \psi(x) = \phi(x) + \frac{C \phi(x)}{|f(x)|^{2n}}
\end{equation*}
with $\delta, C, n$ chosen as in the proof of Proposition \ref{prop:surgery}. As in Remark \ref{remark:both}, by changing $\phi$ we have some $R > \delta$ such that the Liouville vector field of $\psi$ is also outward pointing along $|f|\inv(R)$. To this we wish to add stops as follows:
\begin{itemize}
    \item A stop $F$ given by $f\inv(-R)$;
    \item A stop $\Lambda$ on $|f|\inv(\delta)$ given by the framed Legendrian spheres $\Lambda_i$ along which the additional handles $\ell_i$ are attached.
\end{itemize}
By the generation result \cite[Theorem 1.10]{GPS2}, $\scr{W}(X\setminus f\inv(B_{\delta}), F \cup \Lambda)$ is generated by a collection of two kinds of Lagrangians:
\begin{itemize}
    \item Linking disks of $F$ and of $\Lambda_i$;
    \item Lagrangians $L$ given by the product of cocores of $f\inv(t)$ with an arc $\gamma$ connecting $-R$ to $|f|\inv(\delta)$.
\end{itemize}
Using the wrapping exact triangle of \cite[Theorem 1.9]{GPS2}, we can express the linking disks of $F$ in terms of twisted complexes of the Lagrangians $L$ and the linking disks of $\Lambda_i$, as illustrated in Figure \ref{fig:linking}. Note that the arc $\gamma$ with which we take the product lies below the stops $F,\Lambda$ in the complex plane. 

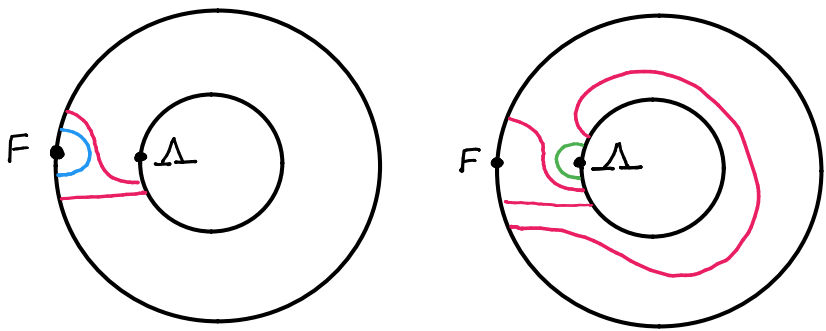
\begin{figure}
\centering
\begin{tikzpicture}[scale=0.5]
\draw[fill=gray!10] (0,0) circle (4);
\draw[fill=white] (0,0) circle (2);
\draw[fill=black] (-2,0) circle (0.1);
\draw[fill=black] (-4,0) circle (0.1);
\draw[thick,red] (-150:2) -- (-165:4);
\path[draw,thick,red,use Hobby shortcut,closed=false] (-160:2).. (180:3) .. (150:3.5) .. (150:4);
\draw[thick,green] (-172.5:4) arc[start angle=-90, end angle=90,radius=0.5];
\draw (-4,0) node[anchor=east] {$F$};
\draw (-2,0) node[anchor=west] {$\Lambda$};
\end{tikzpicture}\begin{tikzpicture}[scale=0.5]
\draw[fill=gray!10] (0,0) circle (4);
\draw[fill=white] (0,0) circle (2);
\draw[fill=black] (-2,0) circle (0.1);
\draw[fill=black] (-4,0) circle (0.1);
\draw[thick,red] (-150:2) -- (-165:4);
\path[draw,thick,red,use Hobby shortcut,closed=false] (-160:2).. (180:3) .. (150:3.5) .. (150:4);
\draw[thick,green] (167.5:2) arc[start angle=90, end angle=270,radius=0.5];
\draw (-4,0) node[anchor=east] {$F$};
\draw (-2,0) node[anchor=west] {$\Lambda$};
\path[draw,thick,red,use Hobby shortcut,closed=false] (160:2) .. (150:2.5) .. (0:3) .. (-90:3.5) .. (-152.5:3.75) .. (-150:4);
\end{tikzpicture}
    \caption{Producing linking disks of $F$ from objects $L$ (in red) and linking disks of $\Lambda$ (in green).\label{fig:linking}}
\end{figure}

Now we pass to $(X \setminus f\inv(t), f)$ by attaching Weinstein handles as in Proposition \ref{prop:surgery}, while keeping the Liouville vector field outward-pointing along $|f|\inv(R)$. Then the linking disks of the $\Lambda_i$ become exactly the cocores $\ell_i$ of the additional handles and so by the gluing result in \cite[Theorem 1.20]{GPS2} we know that $\scr{W}(X \setminus f\inv(t), f)$ is generated by the additional cocores $\ell_i$ and the Lagrangians $L$ (cf. \cite[p.68]{GPS2}). We can take these Lagrangians $L$ to be fibered over an arc $\gamma$ from $-R$ to $f\inv(t)$ that passes below the singular fiber of $f$ and below the stop $F$.  

Now we consider the Weinstein deformation given by:
\begin{equation*}
    \psi_t(x) = \phi(x) + \frac{C \phi(x)}{|f_t(x)|^{2n}}.
\end{equation*}
Arguing similarly to Remark \ref{remark:both}, for all sufficiently small $\epsilon>0$ we can choose $\phi$, larger $C$, $n$, and smaller $R$ and $\delta$ so that there exists some $R\dash >R + \epsilon$ so that all of the points where the Malgrange condition fails for $|f|^2$ lie outside of $|f|\leq R\dash + \epsilon$ and so that
\begin{itemize}
    \item for $0<|t|<R+\epsilon$ the Liouville vector field of $\psi_t$ on $X \setminus f\inv(t)$ is outward pointing along $|f|\inv(r)$ for $R\dash + \epsilon > r \geq R\dash$;
    \item for $R+\epsilon/2 < |t| < R+\epsilon$, the Liouville vector field is outward pointing along $|f|\inv(R)$.
\end{itemize}
 This first property prevents the skeleton of the Liouville domain from escaping outside of a compact set and hence $\psi_t$ gives simple Weinstein deformations of $X \setminus f\inv(t)$ for $0<|t|<R+\epsilon$. Thus the resulting $X\setminus f\inv(t)$ are all exact symplectomorphic and the wrapped Fukaya category $\scr{W}(X\setminus f\inv(t), f)$ remains unchanged (Lemma \ref{lem:deformation}). 

The second property of $\psi_t$ for $R + \epsilon/2 < |t| < R + \epsilon$ means that, since the stop $F$ lives inside $|f|\inv(-R)$, we have a subsector $(f\inv(B_R(0)), f)$ of $(X\setminus f\inv(t), f)$; see Figure \ref{fig:subdomain}. Applying Viterbo restriction, the objects $L$ and $\ell_i$ split-generate $\scr{W}(f\inv(B_R(0)), f)$ (by \cite[Theorem 1.8]{Sylvan2}). But $(f\inv(B_R(0)), f)$ is Weinstein deformation-equivalent to the original sector used to define $\scr{W}(X,f)$ in \S \ref{sec:defns} via the deformation
\begin{equation*}
    \psi_{t,s}(x) = \phi(x) + C \phi(x)\rho_s(f_t(x))
\end{equation*}
for $s \in (0,1]$, where here $\rho_s(z) = \beta_s(|z|)/|z|^{2n} $ for $\beta_s: \RR \to [0,1]$ a smooth cutoff function chosen so that
\begin{align*}
    \beta_s(r) = \begin{cases} 1 & \text{for } r\leq\tan(\uppi s/2) \\
    0 & \text{for } r \geq \tan(\uppi s/2)+\epsilon
    \end{cases}
\end{align*}
and so that $\rho_s(f_t(x))$ is plurisubharmonic (here $\psi_1 = \psi$). Moreover, the Liouville vector field of $\psi_{t,s}$ remains outward pointing along $|f|\inv(r)$ for $R + \epsilon > |t|> r  \geq R+\epsilon/2$ for all $s \in (0,1]$ and so this in fact yields a simple Weinstein deformation. 

Hence we obtain split-generators $L$ and $\ell_i$ of $\scr{W}(X,f)$ (Lemma \ref{lem:deformation}). But the former objects are trivial in this category, since they may be displaced from themselves by a Hamiltonian pushoff to infinity. Hence we obtain the desired split-generation result.
\end{proof}

\begin{figure}
\centering
\begin{tikzpicture}[scale=0.5]
\draw (-5,-4) -- (5,-4) -- (5,4) -- (-5,4) -- (-5,-4);
\draw[fill=black] (-5,0) circle (0.1);
\draw[fill=white] (3,0) circle (0.1);
\draw (-2,0) node {$\huge{\times}$};
\draw (3,0) node[anchor=south] {$t$};
\draw (-5,0) node[anchor=east] {$F$};
\draw (0,-5) node[anchor=north] {$X\setminus f\inv(t)$};
\draw (-2,4) node[anchor=south] {$R$};
\draw[thick,blue] (-2,0) -- (2.9,0);
\draw[thick,red] (-5,-2) to[out=0,in=-90] (2.9,0);
\draw[dashed] (-2,-4) arc[start angle=-60, end angle=60,radius=4.6];
\begin{scope}[shift={(-3,0)}]
\draw[->] (-20:3) -- (-20:4);
\draw[->] (30:3) -- (30:4);
\draw[->] (60:3) -- (60:4.5);
\draw[->] (-70:3.5) -- (-70:4.5);
\end{scope}
\end{tikzpicture}
    \caption{The subdomain constructed in Proposition \ref{prop:generation}: the objects $L$ are in red and the cocores $\ell_i$ in blue. \label{fig:subdomain}}
\end{figure}
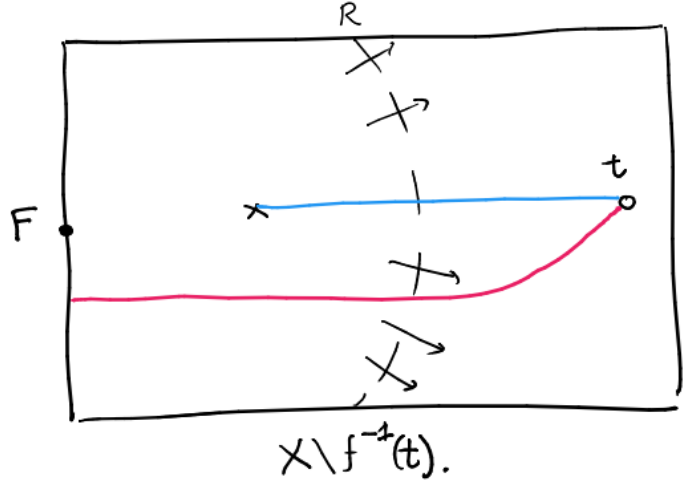

\section{Kn\"orrer Periodicity for Hypersurfaces}\label{sec:proofs}

\subsection{The AAK Equivalence}

We shall begin by proving the smooth case of Kn\"orrer periodicity for hypersurfaces:

\AAK*

First we shall describe the Liouville structures and almost-complex structures we will use on $(X \times \CC, z(f-t))$ to prove this result. We write $F_t = z(f-t)$ in the following. We would like to remind the reader of the conventions in Remark \ref{remark:conventions} concerning the way we construct Liouville structures associated to Landau-Ginzburg models.

Firstly by Proposition \ref{prop:product} since $f\inv(t)$ is smooth for $t\neq 0$, we may \textit{smoothly} identify $V = f\inv(B_{\delta}(t))$ with $f\inv(t) \times \CC$, where $B_{\delta}(t)$ is a small ball around $t \in \CC$. Then in this open set $V \times \CC$, the map $F_t: X \times \CC \to \CC$ is given explicitly by the local Morse-Bott model $f\inv(t) \times  \CC^2 \to \CC$ via $(x,y) \mapsto x y $ where $x,y$ are coordinates in $\CC^2$. 

By Proposition \ref{prop:product}, over an open set $U$ where $f:X \to \CC$ is smoothly trivial, there exists a Liouville deformation $\lambda_t$ to a product Liouville structure on $f\inv(U) \cong f\inv(t) \times U$. We can make this deformation constant outside a larger open set, and since this deformation preserves the fibers this is a simple Liouville homotopy (Definition \ref{def:simple}). We first apply this construction to $U = B_{\delta}(t)$ a small neighbourhood of the smooth fiber $f\inv(t)$ containing no critical points of $f$. Passing to the product $f\inv(B_{\delta}(t)) \times \CC$ with the standard Liouville structure on $\CC$ gives a deformation that remains simple: this deformation gives an isotopy of the fiber of $F_t$ over $-\infty$. By Lemma \ref{lem:deformation} this deformation does not change the wrapped Fukaya category of $(X \times \CC, F_t)$ and so we are free to work with this altered Liouville structure in our proof. Note that this deformation does not change the induced Liouville structure on $f\inv(t)$. 

We apply this argument again, now to the fibration $F_t:X \times \CC \to \CC$: over the subset $\set{\mathrm{Re}(z) > 0} \cup \set{\mathrm{Im}(z) > 0}$ (of course, intersected with a small disk as in \S \ref{sec:defns}) we may again deform the Liouville structure to be of product type, where we take the radial Liouville structure on the base. Note that this deformation leaves the Liouville structure on the smooth fibers $X \setminus f\inv(t)$ unchanged. Once again, we may instead choose to use this Liouville structure on $X \times \CC$ in our arguments.

Now we can make the identification $F_t: f\inv(B_{\delta}(t)) \times \CC \to \CC$ with the map $f\inv(t) \times \CC^2 \to \CC$ given by $(x,y) \mapsto x y $ compatible with the almost-complex structures. Choose a cylindrical almost-complex structure on $f\inv(t)$ and on $X$, and the standard complex structure on $\CC$: we begin by modifying the complex structure on $X$ to agree with the product almost-complex structure on $f\inv(t) \times \CC$ inside $F_{t}\inv(D) \times \CC$ for $D \subset \CC$ a small open disk around $0$, and interpolating with the usual almost-complex structure outside this open subset. Since the Liouville structure is of product type near $f\inv(t)$, the resulting almost-complex structure on $X \times \CC$ can clearly also be chosen to be cylindrical. The same construction could be repeated for any (cylindrical) almost-complex structure on $\CC$.

\begin{proof}(of Theorem \ref{thm:aak})
We begin by choosing a collection $L_i$, $i \in I$ of cylindrical exact Spin Lagrangians inside $f\inv(t)$ representing the cocores of the induced Stein structure on $f\inv(t)$. For each Lagrangian, consider the thimble $T_i$ over $L_i$ given by taking the normalized gradient flow of $\mathrm{Re}(F_t)$ starting along $L_i \subset f\inv(t)$ inside $X \times \CC$. Since each $L_i$ is exact, these $T_i$ are exact Lagrangian submanifolds of $X \times \CC$, and the deformation retraction from $T_i$ to $L_i$ equips them with Spin structures: we claim that they can also be made cylindrical. 

Firstly, there exists some $\epsilon>0$ sufficiently small so that over $B_{\epsilon}(0)$ the thimbles $T_i$ lie entirely inside the Morse-Bott local model $f\inv(B_{\delta}(t)) \times \CC$, by the fact that the exact symplectic structure on $f\inv(B_{\delta}(t))$ is a product. One can check explicitly in the local Morse-Bott model that near $F_{t}\inv(0)$ the $T_i$ are cylindrical at fiberwise $\infty$: they are given by the product of $L_i$ with the standard Lefschetz thimble inside $\CC^2$. Now we need to check at infinity in the base. Outside of the neighbourhood $B_{\epsilon}(0)$ (where all our Lagrangians will intersect) we can instead choose to flow by the Liouville vector field rather than the parallel transport, making the resulting Lagrangian cylindrical at $\infty$.

Now, for each $T_i$ choose a cofinal sequence $T_{i}^{(t)}$ (generically, to ensure transversality of intersections for finite totally ordered sets, see \cite[Definition 3.34]{GPS}) using the time-$t$ flow of the complete Reeb vector field constructed in \cite[Lemma 3.29]{GPS}: by the product form of the Liouville structure over $\set{\mathrm{Re}(z)>0}$, any Reeb flow is given by a rotation in the base: thus we may take a compactly-supported Hamiltonian isotopy to make $T_{i}^{(t)}$ lie over a ray inside $B_{\epsilon}(0)$. Taking $L_{i}^{(t)}$ to be the intersection of $T_{i}^{(t)}$ with the critical locus, by \cite[Lemma 3.29]{GPS}, the $L_{i}^{(t)}$ also form a cofinal sequence, since the quantity:
\begin{equation*}
    \int_{0}^{\infty} \min_{\partial_{\infty} L_t} \alpha(\partial_{\infty} \partial_t L_t) \dd{t}
\end{equation*}
can only increase when restricted to a smaller set. To define $\scr{W}(X \times \CC, F_t)$, choose a collection of Lagrangians with cofinal sequences that includes the $T_{i}^{(t)}$ as a subcollection. 

Now consider the directed Fukaya category $\scr{O}$ given by the $L_{i}^{(j)}$s inside $f\inv(t)$, and the directed Fukaya category $\scr{A}$ given by the $T_{i}^{(t)}$ inside $(X \times \CC, F_t)$, as in \cite[Definition 3.35]{GPS}. We have that for $j>j\dash$,  $\scr{A}(T_{i}^{(j)}, T_{i\dash}^{(j\dash)})$ is given by the Floer complex $CF(T_{i}^{(j)}, T_{i\dash}^{(j\dash)})$: in the base, $T_{i}^{(j)}$ lies over a radial ray that is counterclockwise from $T_{i\dash}^{(j\dash)}$ and hence they intersect only in the fiber $F_{t}\inv(0)$, along the transverse intersection points $L_{i}^{(j)} \cap L_{i\dash}^{(j\dash)}$. Thus $\scr{A}(T_{i}^{(j)}, T_{i\dash}^{(j\dash)}) = \scr{O}(L_{i}^{(j)}, L_{i\dash}^{(j\dash)})$. The additivity of the Maslov index under products, and the fact that $T_{i}^{(j)}$ lies over a radial ray that is strictly counterclockwise from $T_{i\dash}^{(j\dash)}$ means that this is an isomorphism of graded vector spaces. For $j \leq j\dash$ the morphism spaces obviously coincide. Next we shall match the $A_{\infty}$ operations on both directed categories. 

Given a collection of thimbles $T_{i}^{(j_n)}$ for $n=0, \dots, k$ with $j_0 > \cdots > j_k$, choose appropriate (depending on the domain and conformal structure) cylindrical almost-complex structures $J$ for $f\inv(t)$ that make the moduli spaces used to define the operation $\mu^k$ for $L_{i}^{(j_n)}$ regular (as in \cite[p.41]{GPS}), and extend these to cylindrical almost-complex structures on $X \times \CC$ as described above. With respect to such almost-complex structures, the map $F_t: X \times \CC \to \CC$ is $J$-holomorphic in an open set containing all of the intersection points of $L_{i}^{(j_n)}$ (and hence of $T_{i}^{(j_n)}$). Therefore for any holomorphic disk defining the operation $\mu^k$ for the collection $T_{i}^{(j_n)}$, the projection under $F_t$ must be a holomorphic disk in $\CC$ with respect to the standard complex structure (over the entire domain and for every conformal structure). The maximum principle then implies that the holomorphic disk must be contained in $F_{t}\inv(0)$. Another application of the maximum principle with the holomorphic projection $f\inv(t) \times \CC^2 \to \CC^2$ in the local Morse-Bott model shows that all of the holomorphic disks must be contained inside $f\inv(t)$. We claim that these must regular with the chosen almost-complex structures on $X \times \CC$ using the argument from \cite[\S 14c]{SeidelBook}. Because of the product almost-complex structure on the local Morse-Bott model, the linearized Cauchy-Riemann operator splits into a direct sum of linearized Cauchy-Riemann operators on the base $\CC$ and on the fiber $f\inv(t)$. The latter is surjective, by assumption. For the former, since  $T_{i}^{(j_{n+1})}$ always lies over a ray in $B_{\epsilon}(0)$ which is \textit{clockwise} from $T_{i}^{(j_{n})}$, this means that the Maslov index is $0$ and hence the linearized Cauchy-Riemann operator on $\CC$ has index $0$. Then by \cite[Lemma 11.5]{SeidelBook} this linearized Cauchy-Riemann operator on $\CC$ is injective and hence surjective also. Thus the moduli spaces used to define the $A_{\infty}$ operations in $\scr{O}$ and $\scr{A}$ can be made to coincide exactly. Moreover, the discussion of Spin structures from \cite[Corollary 7.8]{AAK} carries over verbatim and shows that the moduli spaces of disks that appear carry the same orientations. Hence we have an \textit{equality} of $A_{\infty}$-categories between $\scr{O}$ and $\scr{A}$. 

Much the same argument as in the previous paragraph carries over to the construction of continuation maps $T_{i}^{(j+1)} \to T_{i}^{(j)}$, so that after localization we have an equivalence of categories between $\scr{W}(f\inv(t))$ and $\scr{A}[C\inv]$. By construction, $\scr{A}[C\inv]$ includes into $\scr{W}(X \times \CC, F_t)$ as a full subcategory, and the composition of this inclusion with the equivalence $\scr{W}(f\inv(t)) \to \scr{A}[C\inv]$ gives the desired functor.

To upgrade $T$ to an equivalence of categories, we will need to know that the category $\scr{W}(X \times \CC, z(f-t))$ is generated by thimbles. Since in a neighbourhood of the critical locus of the function $z(f-t)$ the Stein structure can be deformed into a product, the cocores of the index-$n$ critical points of the deformed Stein function will be given by products of cocores of critical handles for $f\inv(t)$ with the ordinary Lefschetz thimble of $z_1 z_2$ on $\CC^{2}$. By Proposition \ref{prop:generation}, these Morse-Bott thimbles are split-generators of $\scr{W}(X,z (f-t))$. Compare also \cite{SeidelNotes,AG} for Morse-Bott generation by thimbles.
\end{proof}

Now we turn to the relative version of this theorem:

\relAAK*

The second category here is \textit{fiberwise stopped} with respect to $g$ as in \cite{AA}, where admissible Lagrangians are fibered over arcs under the original map $f:X \to \CC$ while also being admissible in the fibers of $f$ for the restriction $g|_{f\inv(t)}$.

\begin{definition}\label{defn:relative}
Given $X$ a Stein manifold with $f,g:X \to \CC$ two holomorphic functions, we define the \textbf{relative Fukaya-Seidel category} $\scr{W}(X, f, g)$ to be the wrapped Fukaya category $\scr{W}(X, F \cup G)$ of $X$ with \textbf{relative stop} given by the union of:
\begin{itemize}
    \item $F$ the relative skeleton of the sector $(f\inv(-\infty), g)$, inside $\partial^{\infty}X$;
    \item $G$ the union of $(g|_{f\inv(t)})\inv(-\infty) \subset \partial^{\infty} f\inv(t) \subset \partial^{\infty}X$ for $t$ along an arc connecting $-\infty$ to $0$;
\end{itemize}
as illustrated in Figure \ref{fig:relative}.
\end{definition}

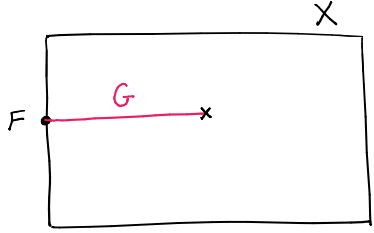
\begin{figure}
\centering
\begin{tikzpicture}[scale=0.4]
\draw (-5,-4) -- (5,-4) -- (5,4) -- (-5,4) -- (-5,-4);
\draw[fill=black] (-5,0) circle (0.1);
\draw (0,0) node {$\huge{\times}$};
\draw (-5,0) node[anchor=east] {$F$};
\draw (-2,4) node[anchor=south] {$X$};
\draw[red] (-2,0) node[anchor=south] {$G$};
\draw[thick,red] (-4.9,0) -- (0,0);
\end{tikzpicture}
    \caption{The relative stop. \label{fig:relative}}
\end{figure}

This stop can be thought of as the boundary at infinity of the relative skeleton of the relative skeleton. In the sequel, we shall prove that under strong assumptions on $f, g$ and $X$, the relative Fukaya-Seidel category is indeed equivalent to the Fukaya-Seidel category $\scr{W}(X, f + \delta g)$ for $\delta>0$ sufficiently small. This equivalence certainly should not hold in general.

Because of this construction, the proof of Theorem \ref{thm:aak} carries over exactly to show that taking thimbles over admissible Lagrangians in the critical locus gives a fully faithful functor
\begin{equation*}
    T: \scr{W}(f\inv(t),g) \to \scr{W}(X \times \CC, z(f-t), g)
.\end{equation*}
We expect this to be an equivalence under an appropriate modification of Proposition \ref{prop:generation}.

\subsection{Proof of Main Theorem}

Since there are two Landau-Ginzburg models under consideration, it will be important to establish notation for the corresponding cap and cup functors:
\begin{itemize}
    \item For the Landau-Ginzburg model $(X,f)$ we shall use $\cap: D\scr{W}(X,f) \to D\scr{W}(f\inv(t))$ and $\cup: D\scr{W}(f\inv(t)) \to D\scr{W}(X,f)$;
    \item For the Landau-Ginzburg model $(X \times \CC,  F_t)$ we shall use instead the notation $\bigcap: D\scr{W}(X \times \CC, F_t) \to D\scr{W}(X\setminus f\inv(t))$ and $\bigcup: D\scr{W}(X\setminus f\inv(t)) \to D\scr{W}(X \times \CC, F_t)$.
\end{itemize}
We will regard objects of $D\scr{W}(X\setminus f\inv(t),f)$ as living inside $D\scr{W}(X,f)$ via the Viterbo restriction map constructed in Proposition \ref{prop:generation}.

Our result will follow by showing that under the AAK equivalence in Theorem \ref{thm:aak}, the cap functor $\cap: \scr{W}(X,f) \to \scr{W}(f\inv(t))$ applied to the cocore $\ell \subset X$ of an additional handle, is isomorphic in $\scr{W}(X\times \CC, z(f-t))$ to the linking disk of the core of this additional handle for the Liouville structure on $X\setminus f\inv(t)$. By \cite[\S 7.2]{GPS2}, the linking disks of these handles are given by the cup functor $\bigcup: \scr{W}(X\setminus f\inv(t)) \to \scr{W}(X\times \CC, z(f-t))$ applied to the cocores $\ell$ of the handles. Therefore this isomorphism is exactly \cite[Lemma A.28]{abouzaidSmith}: unfortunately the proof sketched therein is difficult to make rigorous. The isomorphism would also follow from largely formal arguments involving the adjunctions constructed in \cite{AG}. Here we provide an alternative argument independent of \cite{AG}.  

Using the Yoneda lemma for $A_{\infty}$-categories \cite[Corollary 14.7]{SeidelBook} what we wish to show is that:

\begin{theorem}\label{thm:cones}
For every additional cocore $\ell$, we have an equivalence of left modules over $\scr{W}(f\inv(t))$:
\begin{equation*}
    T^{\ast} \scr{Y}_{\bigcup(\ell)} \cong \scr{Y}_{\cap \ell}
\end{equation*}
where here $\scr{Y}: \scr{W}(f\inv(t)) \to \scr{W}(f\inv(t))\mathrm{Mod} $ is the left Yoneda functor, and $T^{\ast}$ denotes the pullback of left modules.
\end{theorem}
\begin{proof}
For this proof we will continue to use the Liouville structures constructed on $X \times \CC$ above; note that with respect to the induced Liouville structure, every fiber of $F_t$ contains an open neighbourhood diffeomorphic to $f\inv(t) \times \CC^{\ast}$ where the Liouville vector field is given by the product with the standard Liouville vector field given by:
\begin{equation*}
    Z = -\frac{\epsilon}{r^3} \partial_{r}
\end{equation*}
where $r = |z|$ is the radial coordinate. 

Given any Lagrangian $L$ in $\scr{W}(f\inv(t))$, the thimble $TL$ and $\bigcup \ell$ intersect in a single fiber of $F_t$, say over $0 < \delta \in \CC$, where we take $\delta>0$ sufficiently small so that over $B_{\delta}$, the Lagrangian $TL$ will be contained entirely inside the Morse-Bott model. Because of the product structure, in the fiber $F_{t}\inv(\delta)$ the Lagrangian $TL$ is given by $L \times K \subset f\inv(t) \times \CC^{\ast} \subset X\setminus f\inv(t)$ for $K \subset \CC^{\ast}$ a small circle around $0$. By our definition of the cup functor, $\bigcup \ell$ in this fiber $F_{t}\inv(\delta)$ is given exactly by the cocore $\ell$; see Figure \ref{fig:modules1}.

Observe that in the analysis of the Stein structure on $X\setminus f\inv(t)$ in Proposition \ref{prop:surgery} every additional cocore must have boundary in $f\inv(t)$: they cannot remain in a compact set, since there are no index-$(n+1)$ critical points of the Stein function, and they cannot have boundary in $|f -t|\inv(\delta)$ because the Liouville vector field is inward-pointing. The Liouville structure we use on $X \times \CC$ to define the Fukaya category is a deformation of this Stein structure, but this exact deformation preserves flow lines of the Liouville vector field. Because of the local form of the Liouville vector field on $f\inv(t) \times \CC^{\ast} \subset X\setminus f\inv(t)$ this means that $\ell$ is fibered over $\CC^{\ast}$ and hence intersects $L \times K$ in a single fiber of $f$. Inside this fiber, $\bigcup \ell$ is given by $\cap \ell$, by our definition of the cap functor, and hence $TL$ and $\bigcup \ell$ intersect exactly along $L \cap (\cap \ell) \subset f\inv(t)$; see Figure \ref{fig:modules1}. Thus we have an isomorphism of vector spaces:
\begin{equation*}
    T^{\ast} \scr{Y}_{\bigcup(\ell)}(L) = CF\br{\bigcup \ell, TL} \cong CF(\cap \ell, L) = \scr{Y}_{\cap \ell}(L)
.\end{equation*}
Moreover, by additivity of the Maslov index under products, this is in fact an isomorphism of graded vector spaces. It remains to match the $A_{\infty}$ module operations $\mu^k$.

\begin{figure}
\centering
\begin{tikzpicture}[scale=0.75]
\draw[thick] (-5,-4) to (3,-4) to (5,0) to (-3,0) to (-5,-4);
\draw (0,-2) node {$\huge{\times}$};
\draw[fill=black] (2,-2) circle (0.05);
\draw[dashed] (2,-2) -- (2,1);
\draw[dashed] (0,-2) ellipse (3 and 1.5);
\draw[red,thick] (0,-2) -- (4,-2);
\draw (4,-2) node[anchor=west] {$T(L)$};
\draw[fill=black] (2,-2) circle (0.05);
\draw[fill=black] (-4,-2) circle (0.05);
\draw (-4,-2) node[anchor=east] {$F$};
\path[draw,thick,blue,use Hobby shortcut,closed=false] (-2,0) .. (-2,-2) .. (0,-3) .. (2,-2) .. (3,0);
\draw (-2,0) node[anchor=south] {$\cup \ell$};
\begin{scope}[shift={(2,3)},scale=0.5]
\draw[thick] (-5,-4) to (3,-4) to (5,0) to (-3,0) to (-5,-4);
\draw (-2,-2) node {$\huge{\times}$};
\draw[fill=white] (2,-2) circle (0.1);
\draw[thick,blue] (-2,-2) to (1.9,-2);
\draw[dashed] (2,-2) ellipse (1.9 and 1);
\draw[thick,red] (2,-2) ellipse (1 and 0.5);
\draw[fill=black] (1,-2) circle (0.1);
\draw[<-] (1,-2) -- (1,1);
\draw (1,1) node[anchor=south] {$TL \cap (\cup \ell)$};
\end{scope}
\end{tikzpicture}
    \caption{Identifying the modules as graded vector spaces.\label{fig:modules1}}
\end{figure}
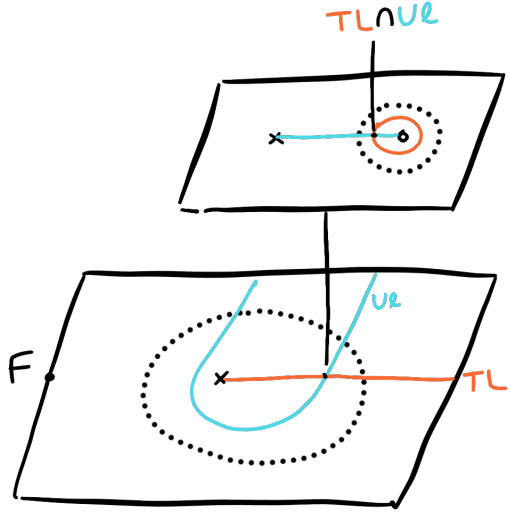

We begin with the differential $\mu^1$. In this case, the intersection of $TL$ and $\bigcup \ell$ is contained entirely inside the fiber $f\inv(t)$ inside the fiber $F_{t}\inv(\delta)$. An iterated version of the maximum principle argument in Theorem \ref{thm:aak} applied to the holomorphic maps $F_t$ and $f$ shows that the holomorphic disks computing $\mu^1: T^{\ast} \scr{Y}_{\bigcup(\ell)}(L) \to T^{\ast} \scr{Y}_{\bigcup(\ell)}(L)$ are all contained in this fiber $f\inv(t) \subset F_{t}\inv(\delta)$; and by a variation of the transversality argument there, we see that any almost-complex structures used to achieve transversality for disks in the fiber $f\inv(t)$ can be extended to $X \times \CC$ so as to make the disks in the total space for $\mu^1: T^{\ast} \scr{Y}_{\bigcup(\ell)}(L) \to T^{\ast} \scr{Y}_{\bigcup(\ell)}(L)$ regular also. 

Now we consider the product operation $\mu^2$ as a simpler prototype of the argument to follow; suppose $L_1, L_2$, are exact cylindrical Lagrangians in $f\inv(t)$, and let $T^i L_{(i)} = T_{i}^{(i)}$ be their thimbles as in the proof of Theorem \ref{thm:aak}. Suppose we have intersection points $y_1 \in (\cap \ell) \cap L_1$, $x \in L_1 \cap L_{2}$ and $y_2 \in T^{(2)}L_2 \cap (\cap \ell)$, and that we have chosen regular almost-complex structures $J$ on $f\inv(t)$ (domain and conformal-structure dependent) for the moduli spaces $\mathscr{M}^{f\inv(t)}_{J}(y_1, x, y_2)$ of disks defining the operation:
\begin{equation*}
     CF(L_{2}, L_{1}) \otimes \scr{Y}_{\cap \ell}(L_{2})  \to \scr{Y}_{\cap \ell}(L_1)
.\end{equation*}
Extend these almost-complex structures to ones on $X \times \CC$ that are of product type on the neighbourhood $f\inv(t) \times \CC$ as described above, which we shall also denote by $J$.

By Theorem \ref{thm:aak} and the above argument, these intersection points correspond to intersection points $y_1 \in \bigcup \ell \cap T^{(1)} L_1$ living in $F_{t}\inv(\delta)$, an intersection point $x \in T^{(1)}L_1 \cap T^{(2)}L_{2}$ lying in $F_{t}\inv(0)$, and $y_2 \in T^{(2)}L_2 \cap \bigcup \ell$, living in $F_{t}\inv(\e{i\theta} \delta)$ for $0< \theta<\pi$ by our construction of the thimbles $T_{i}^{(i)}$. Consider now the moduli space of holomorphic disks $\mathscr{M}^{X \times \CC}_{J}(y_1, x, y_2)$ inside $X \times \CC$ associated to the operation
\begin{equation*}
     CF(T^{(2)}L_2, T^{(1)}L_1) \otimes \scr{Y}_{\bigcup \ell}(T^{(2)} L_{2}) \to \scr{Y}_{\bigcup \ell}(T^{(1)}L_1)
.\end{equation*}
All of these intersection points lie inside the local Morse-Bott model $f\inv(t) \times \CC^2$ and so by the maximum principle for the standard almost-complex structure on $\CC$ applied to the projections under the holomorphic maps $F_t$ and $f$, all of these disks must lie inside the local Morse-Bott neighbourhood $f\inv(t) \times \CC^2 \subset F_{t}\inv(B_{\delta})$. Inside this Morse-Bott model, we have various holomorphic projections. Firstly, projection to $f\inv(t)$ gives a surjective map
\begin{equation*}
    p: \mathscr{M}^{X \times \CC}_{J}(y_1, x, y_2) \to \mathscr{M}^{f\inv(t)}_{J}(y_1, x, y_2)
\end{equation*}
because of how the almost-complex structures have been defined. We claim that this is in fact a bijection, and that $\mathscr{M}^{X \times \CC}_{J}(y_1, x, y_2)$ is regular with respect to this almost-complex structure. To see this, suppose we have such a disk $(u, r) \in \mathscr{M}^{X \times \CC}_{J}(y_1, x, y_2)$ for $u:D \to X \times \CC$ and $r$ the \textit{unique} conformal class of $3$-pointed disks $D$: then we can show that it is determined by its projection to $f\inv(t)$ as follows.

Let $\pi: f\inv(t) \times \CC^2 \to \CC^2$ be the (holomorphic) projection to the second factor. We consider the holomorphic map $\pi \circ u$ from a 3-pointed disk $D$ to $\CC^2$. Since there is a unique conformal equivalence class of 3-pointed disk, this holomorphic disk $(\pi \circ u, r)$ computes the operation $\mu^2: CF(T_{2}, T_1) \times CF(U,T_{2})  \to CF(U, T_1)$ where $T_i,U \subset \CC^2$ are the Lagrangians obtained by the projections $\pi(T_{(i)}L_i), \pi(\bigcup \ell)$, respectively; see Figure \ref{fig:modules2}. These $T_i$ are the thimbles for the standard Lefschetz fibration on $\CC^2$ and hence intersect transversely at the origin. The Lagrangian $U$ is the parallel transport of the real axis $\RR_{+}$ inside the generic fiber $\set{z_1 z_2 = \delta} \cong \CC^{\ast}$ of the model Lefschetz fibration, and so intersects each of the $T_i$ transversely in a single point.

For a generic almost-complex structure $J$ on $\CC^2$, this moduli space will consist of a single regular disk computing the product of the unique point $p \in U \cap T_{1}$ with the identity of $T_1 \cong T_{2}$. Applying the previous construction of an almost-complex structure on $X \times \CC$ with this $J$ in place of the standard almost-complex structure on $\CC^2$, we see that we can ensure that for any $(u,r) \in \mathscr{M}^{X \times \CC}_{J}(y_1, x, y_2)$, both $p(u,r)$ and $\pi(u,r)$ are regular. But this shows exactly that there is a unique holomorphic disk in the preimage under $p$ of any disk in $\mathscr{M}^{f\inv(t)}_{J}(y_1, x, y_2)$. It is straightforward to check that the induced orientations on these moduli spaces agree.

\begin{figure}
\centering
\begin{tikzpicture}[scale=0.7]
\fill[fill=green!10] (0,0) -- (2,0) -- (2,1) -- (0,0); 
\draw (-5,-4) -- (5,-4) -- (5,4) -- (-5,4) -- (-5,-4);
\draw (0,-4.5) node[anchor=north] {$X\times \CC$};
\draw (0,0) node {$\huge{\times}$};
\draw[thick,red] (0,0) -- (5,0);
\draw[thick,red] (0,0) -- (2,1) -- (5,2.5);
\draw[dashed] (0,0) circle (3);
\draw[thick,blue] (-2,0) arc[start angle=-180, end angle=0,radius=2];
\draw[thick,blue] (-2,0) -- (-2,4);
\draw[thick,blue] (2,0) -- (2,4);
\draw[fill=black] (2,0) circle (0.05);
\draw[fill=black] (2,1) circle (0.05);
\draw (5,0) node[anchor=west] {$TL_1$};
\draw (5,2.5) node[anchor=west] {$TL_2$};
\draw (2,4.5) node[anchor=east] {$\cup \ell$};
\draw (0,0) node[anchor=north east] {$x$};
\draw (2,0) node[anchor=north west] {$y_1$};
\draw (2,1) node[anchor=south east] {$y_2$};
\end{tikzpicture}
    \caption{Identifying the $\mu^2$ operations on the modules: the holomorphic disk is in green. \label{fig:modules2}}
\end{figure}
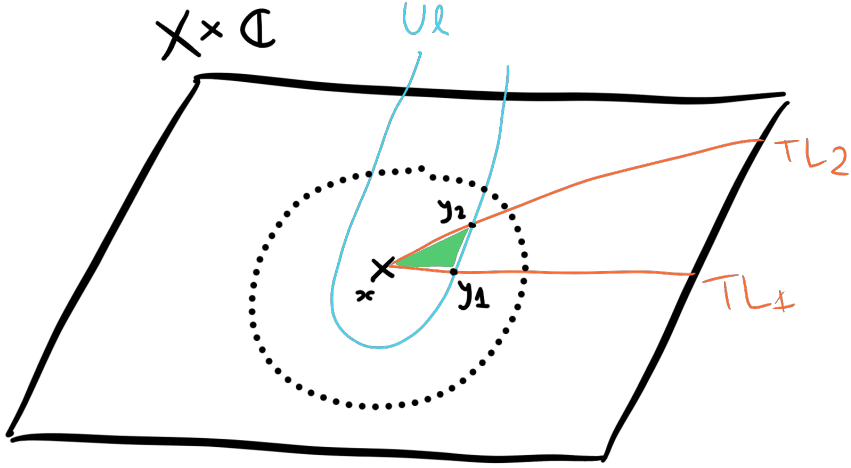


For the higher $\mu^k$ operations, we begin by generalizing the above argument for $\mu^2$. Suppose $L_i$, $i=1, \dots, k$ are exact cylindrical Lagrangians in $f\inv(t)$, and let $T^{(i)}L_i = T_{i}^{(i)}$ be their thimbles as in the proof of Theorem \ref{thm:aak}. Suppose we have intersection points $y_1 \in (\cap \ell) \cap L_1$, $x_i \in L_i \cap L_{i+1}$ for $i=1, \dots, k-1$ and $y_2 \in T^{(k)}L_k \cap (\cap \ell)$, and that we have chosen regular almost-complex structures $J$ on $f\inv(t)$ (domain and conformal-structure dependent) for the moduli space $\mathscr{M}^{k, f\inv(t)}_{J}(y_1, x_1, \dots, x_k, y_2)$ of disks defining the operation:
\begin{equation*}
    CF(L_2, L_1) \otimes \cdots CF(L_k,L_{k-1}) \otimes  \scr{Y}_{\cap \ell}(L_{k}) \to \scr{Y}_{\cap \ell}(L_1)
.\end{equation*}
Extend these almost-complex structures to ones on $X \times \CC$ that are of product type on the neighbourhood $f\inv(t) \times \CC$ as described above. 

By Theorem \ref{thm:aak} and the above argument, these intersection points correspond to intersection points $y_1 \in \bigcup \ell \cap T^{(1)} L_1$ living in $F_{t}\inv(\delta)$, intersection points $x_i \in T^{(i)}L_i \cap T^{(i+1)}L_{i+1}$ for $i=1, \dots, k-1$ lying in $F_{t}\inv(0)$ and $y_2 \in T^{(k)}L_k \cap \bigcup \ell$, living in $F_{t}\inv(\e{i\theta} \delta)$ for $0< \theta<\pi$ by our construction of the thimbles $T_{i}^{(i)}$. Consider now the moduli space of holomorphic disks $\mathscr{M}^{k, X \times \CC}_{J}(y_1, x_1, \dots, x_k, y_2)$ inside $X \times \CC$ associated to the operation
\begin{equation*}
     CF(T^{(2)}L_2, T^{(1)}L_1) \otimes \cdots \otimes CF(T^{(k)} L_{k}, T^{(k-1)} L_{k-1}) \otimes  \scr{Y}_{\bigcup \ell}(T^{(k)} L_{k}) \to \scr{Y}_{\bigcup \ell}(T^{(1)}L_1)
.\end{equation*}
All of these intersection points lie inside the local Morse-Bott model $f\inv(t) \times \CC^2$ and so by the maximum principle for the standard almost-complex structure on $\CC$ applied to the projections under the holomorphic maps $F_t$ and $f$, all of these disks must lie inside the local Morse-Bott neighbourhood $f\inv(t) \times \CC^2 \subset F_{t}\inv(B_{\delta})$. Inside this Morse-Bott model we have various holomorphic projections. Firstly, projection to $f\inv(t)$ gives a surjective map
\begin{equation*}
    p: \mathscr{M}^{k, X \times \CC}_{J}(y_1, x_1, \dots, x_k, y_2) \to \mathscr{M}^{k, f\inv(t)}_{J}(y_1, x_1, \dots, x_k, y_2)
\end{equation*}
because of how the almost-complex structures have been defined. We claim that this is in fact a bijection, and that $\mathscr{M}^{k, X \times \CC}_{J}(y_1, x_1, \dots, x_k, y_2)$ is regular with respect to this almost-complex structure (as before, one can see that the induced orientations on these moduli spaces agree). To demonstrate this, suppose we have such a disk $(u, r) \in \mathscr{M}^{k, X \times \CC}_{J}(y_1, x_1, \dots, x_k, y_2)$ for $u:D \to X \times \CC$ and $r$ a \textit{fixed conformal class} of marked disks: then we can show that it is determined by its projection to $f\inv(t)$ as follows.

Let $\pi: f\inv(t) \times \CC^2 \to \CC^2$ be the (holomorphic) projection to the second factor. We consider the holomorphic map $\pi \circ u$ and we again wish to show that there is a unique holomorphic disk in the preimage under $p$ of any disk in $\mathscr{M}^{k, f\inv(t)}_{J}(y_1, x_1, \dots, x_k, y_2)$. Now since the conformal class $r$ is fixed, the disk $\pi \circ u$ no longer computes the $A_{\infty}$ operation $\mu^k$. This holomorphic disk $(\pi \circ u, r)$ still has boundary lying along the Lagrangians $T_i,U \subset \CC^2$ obtained by the projections $\pi(T^{(i)}L_i), \pi(\bigcup \ell)$, respectively, as above; and has corners at the intersection points $p_1 \in U \cap T_1, p_k \in U \cap T_k$ and the identity elements $a_i \in T_{i} \cap T_{i+1}$. 

Now choose $\gamma: [0,1) \to \bar{\scr{S}}_{k}$ a smooth simple path of conformal structures in the compactified moduli space of $(k+1)$-pointed disks that travels from the fixed conformal structure $r = \gamma(0)$ towards the deepest boundary stratum $r_0 = \lim_{t \to 1^{-}}\gamma(t)$ where $r$ degenerates in to a one-legged tree of $3$-pointed disks $r_0$ (see Figure \ref{fig:disk_tree}). We now consider the parametrized moduli space $\tilde{\mathscr{M}}^{k, \CC^2}_J(p_1, a_{1}, \dots, a_{k-1}, p_k)$ of $J$-holomorphic disks $(v,s)$ in $\CC^2$ with the same boundary conditions as $u$, but with conformal structure $s$ on the domain of $v$ allowed to vary along the path $\gamma$. The standard arguments show that for a generic almost-complex structure $J$ on $\CC^2$, this moduli space becomes a smooth $1$-manifold. Again applying the previous construction of an almost-complex structure on $X \times \CC$ with the modified $J$ in place of the standard almost-complex structure on $\CC^2$, we see that we can ensure that for any $(u,r) \in \mathscr{M}^{k, X \times \CC}_{J}(y_1, x_1, \dots, x_k, y_2)$, both $p(u,r)$ and $\pi(u,r)$ are regular also.

\begin{figure}
\centering
\begin{tikzpicture}[scale=0.4]
\begin{scope}[rotate=-45]
\filldraw[fill=gray!10] (0,0) circle (3);
\draw[fill=black] (60:3) circle (0.1);
\draw[fill=black] (-60:3) circle (0.1);
\begin{scope}[shift={(-5,0)}]
\filldraw[fill=gray!10] (0,0) circle (2);
\draw[fill=black] (0:2) circle (0.1);
\draw[fill=black] (180:2) circle (0.1);
\draw[fill=black] (-90:2) circle (0.1);
\end{scope}
\begin{scope}[shift={(-12,0)}]
\filldraw[fill=gray!10] (0,0) circle (2);
\draw[fill=black] (0:2) circle (0.1);
\draw[fill=black] (120:2) circle (0.1);
\draw[fill=black] (-120:2) circle (0.1);
\end{scope}
\draw[dashed] (-8.5,0) circle (1.5);
\draw[fill=black] (-8.5,-1.5) circle (0.1);
\draw (-8.75,0) node {$\ddots$};
\end{scope}
\end{tikzpicture}
    \caption{A one-legged tree of three-pointed holomorphic disks.\label{fig:disk_tree}}
\end{figure}
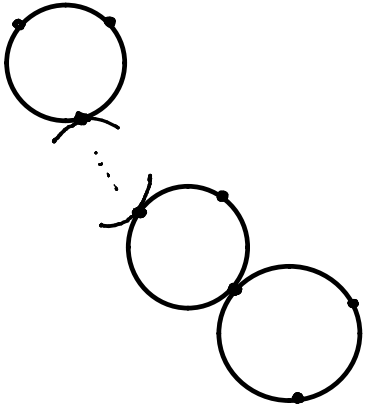

Now we consider the Gromov compactification of $\tilde{\mathscr{M}}^{k, \CC^2}_J(p_1, a_{1}, \dots, a_{k-1}, p_k)$. By the maximum principle applied to the fibration $\CC^2 \to \CC$, disks in this moduli space may not escape to infinity and so we may apply Gromov compactness. Exactness prohibits disk and sphere bubbling, and since all of $p_1, a_i, p_k$ are in the same degree, there can be no strip breaking. Thus the only boundary components of $\tilde{\mathscr{M}}^{k, \CC^2}_J(p_1, a_{1}, \dots, a_{k-1}, p_k)$ come from the conformal structures of the domain at either end of the path $\gamma$. At $0$, the conformal structure is $\gamma(0) = r$ and hence one part of the Gromov boundary of $\tilde{\mathscr{M}}^{k, \CC^2}_J(p_1, a_{1}, \dots, a_{k-1},p_k)$ is given by the moduli space of disks $(u,r)$ with fixed conformal structure $r$. On the other end, we have $\lim_{t \to 1^{-}}\gamma(t) = r_0$; this component of the boundary consists of disks that contribute to the iterated product operation:
\begin{equation*}
   \mu^2(\mu^2(\cdots \mu^2(a_1,a_2) \cdots), a_{k-1}), p_k)  = p_1
\end{equation*}
of $p_k \in U \cap T_{k}$ with the identity $a_i$ of $T_i \cong T_{i+1}$. As above, there is a unique holomorphic disk in $\CC^2$ representing each one of these product operations, and thus there is only one such disk. Hence, since $\tilde{\mathscr{M}}^{k, \CC^2}_J(p_1, a_{1}, \dots, a_{k-1}, p_k)$ is a $1$-manifold with boundary, there must be exactly one disk $(u,r)$ with conformal class $r$ under the projection $\pi$ to $\CC^2$. 
\end{proof}

We may now prove the main result.

\mainthm*

\begin{proof}
This follows by combining our previous results. To spell out the argument explicitly: 
\begin{itemize}
    \item By Corollary \ref{cor:alt_defn}, $\scr{W}(f\inv(0))$ is the quotient of $\scr{W}(f\inv(t))$ by the subcategory $\scr{D}\dash$ generated by Lagrangians of the form $\cap L$;
    \item By Proposition \ref{prop:generation}, this subcategory $\scr{D}\dash$ is the same as the subcategory generated by the Lagrangians $\cap \ell$ for $\ell$ additional cocores;
    \item By Theorem \ref{thm:aak}, $\scr{W}(f\inv(t))$ is equivalent to $\scr{W}(X \times \CC, z(f-t))$, and under this equivalence, generators $\cap \ell$ of the subcategory $\scr{D}\dash$ are sent to the objects $\bigcup \ell$ of the subcategory $\scr{D}$ by Theorem \ref{thm:cones};
    \item By Proposition \ref{prop:stops}, the quotient of $\scr{W}(X \times \CC, z(f-t))$ by the subcategory $\scr{D}$ yields $\scr{W}(X \times \CC, z f)$. Thus:
\end{itemize}
\begin{equation*}
    D^{\pi}\scr{W}(f\inv(0)) \cong D^{\pi}\scr{W}(f\inv(t))/\scr{D}\dash \cong D^{\pi} \scr{W}(X \times \CC, z(f-t))/\scr{D} \cong D^{\pi}\scr{W}(X \times \CC, z f)
.\end{equation*}
\end{proof}

\section{Applications to Mirror Symmetry}\label{sec:hms}

For the following applications to mirror symmetry we will make use of the following lemma from algebraic geometry. The proof is fairly simple and a sketch can be found in  \cite[p.5]{Subalgebras}. Suppose $X$ is a smooth quasiprojective algebraic variety, $L$ is a line bundle with a section $s$, and let $U = X \setminus s\inv(0)$. 

\begin{lemma}\label{lemma:complement}
Let $s: L\inv \otimes \to \mathrm{id}$ be the natural transformation given by multiplying by the section $s$. Then localizing at $s$ gives an equivalence of categories:
\begin{equation*}
    D^b \mathrm{Coh}(X)[s\inv] \cong D^b \mathrm{Coh}(U).
\end{equation*}
\end{lemma}

A heuristic way to interpret this lemma is that localization gives the connection between removing a divisor on one side of mirror symmetry and degenerating to a special fiber on the other.

\subsection{The Tower of Pants}\label{sec:pants}

The Landau-Ginzburg model $(\CC^{n+1}, W_n= z_1 \cdots z_{n+1})$ has particular importance in mirror symmetry as it arises as the mirror to the higher-dimensional pair of pants which we denote by $\Pi_{n-1} = \set{x_1 + \dots + x_n + 1 = 0} \subset (\CC^{\ast})^{n}$: see for instance \cite{NadlerI, AAK}. In this section, we will sketch another proof of mirror symmetry and periodicity in this special case, drawing upon work of Abouzaid-Auroux \cite{AA}.

In the case of this LG model, the general fiber $W_n\inv(1)$ is given by the complex torus $(\CC^{\ast})^{n}$, for which we know that the wrapped Fukaya category is generated by the single Lagrangian $L = (\RR_{+})^{n} \subset (\CC^{\ast})^{n}$, with endomorphism algebra given by $\CC[x_{1}^{\pm}, \dots, x_{n}^{\pm}]$. We shall use the following results proved by Abouzaid-Auroux:

\begin{theorem*}(Abouzaid-Auroux \cite{AA})
The category $\scr{W}(\CC^{n+1}, z_1 \cdots z_{n+1})$ is generated by the single Lagrangian $\cup L$. Moreover, Seidel's natural transformation $s$ on the general fiber of this LG model corresponds to a morphism $L \to L$ given by multiplication by $x_1 + \dots + x_{n} + 1$. 
\end{theorem*}
Each of the terms in the sum $x_1 + \cdots + x_n+1$ in the natural transformation corresponds to a count of holomorphic curves passing through one of the hyperplanes $z_i=0$ in the singular fiber $z_1 \dots z_{n+1} =0$ of the LG model. 

\begin{corollary}
We have quasiequivalences of categories:
\begin{equation*}
    D^{\pi}\scr{W}(W_{n}\inv(0))  \simeq  D^{\pi}\scr{W}(\CC^{n+2}, z_1 \dots z_{n+2}) \simeq D^b\mathrm{Coh}(\Pi_n)
.\end{equation*}
\end{corollary}
\begin{proof}
Using the localization definition, the first category is given by the localization of the category of $\CC[x_{1}^{\pm}, \dots, x_{n}^{\pm}]$-modules at the natural transformation $\mathrm{id} \to \mathrm{id}$ given by multiplication by $x_1 + \dots + x_{n} + 1$. This is the same as the category of modules over the localization of $\CC$-algebras of $\CC[x_{1}^{\pm}, \dots, x_{n}^{\pm}]$ at $x_1 + \dots + x_{n} + 1$, that is, the category of coherent sheaves on $\set{x_1 + \dots + x_n + 1 \neq 0} \subset (\CC^{\ast})^n$: but this is exactly $\set{x_1 + \dots + x_n  + 1 = -x_{n+1}} \subset (\CC^{\ast})^{n+1}$, the pair of pants $\Pi_{n}$.

The second category can be given by computing the endomorphisms of the generator $\cup L$. Since $\cup L$ is given by applying a $\cup$ functor, we may push it off itself to see that it has endomorphism algebra given by the cone of the map $\mathrm{End}(L) \to \mathrm{End}(L)$ given by multiplication by the natural transformation $x_1 + \dots + x_{n+1} + 1$ (see Figure \ref{fig:hms}). Hence $\scr{W}(\CC^{n+2}, z_1 \dots z_{n+2})$ is given by the category of modules over the quotient of $\CC[x_{1}^{\pm}, \dots, x_{n+1}^{\pm}]$ by the ideal $(x_1 + \dots + x_{n+1} + 1)$. But this is exactly the category of coherent sheaves on $\set{x_1 + \dots + x_{n+1} + 1 = 0} \subset (\CC^{\ast})^{n+1}$, that is, the pair of pants $\Pi_{n}$.
\end{proof}

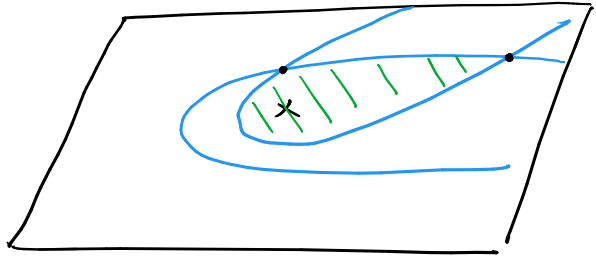
\begin{figure}
\centering
\begin{tikzpicture}[scale=0.5]
\filldraw[thick,blue,fill=green!10] (0.5,-1.9) to[out=5,in=-90] (2,2) to[out=170,in=90] (-1,0) to[out=-90,in=-180] (0.5,-1.9);
\draw (-5,-4) -- (5,-4) -- (5,4) -- (-5,4) -- (-5,-4);
\draw (0,0) node {$\huge{\times}$};
\draw[thick,blue] (-2,4) -- (-2,0) to[out=270,in=180] (0.5,-1.9);
\draw[thick,blue] (2,2) -- (2,4);
\draw[thick,blue] (2,2) -- (5,2);
\draw[thick,blue] (0.5,-1.9) -- (5,-1.9);
\draw[fill=black] (2,2) circle (0.1);
\draw[fill=black] (0.5,-1.9) circle (0.1);
\end{tikzpicture}
    \caption{Calculating endomorphisms of a $\cup$-shaped object.\label{fig:hms}}
\end{figure}

It should be possible to extend this proof to complements of hypersurfaces in toric varieties, following the ideas of Abouzaid-Auroux, as follows. Suppose $H \subset V$ is a hypersurface in a toric variety $V$ defined by a section $s \in \Gamma(V, \scr{O}(H))$. The mirror to $H \subset (\CC^{\ast})^n$ is given by a toric LG model $(Y, W_Y)$ \cite{Speculations}; compactifying $H \subset (\CC^{\ast})^n$ to $H \subset V$ adds extra terms to the superpotential, so that the mirror is $(Y,W_Y + \delta W)$: here $W$ is considered as a secondary superpotential used to wrap fiberwise, as in \S \ref{sec:proofs}. 

\begin{theorem*}(Abouzaid-Auroux \cite{AA})
The Fukaya-Seidel category $D^{\pi}\scr{W}(W_{Y}\inv(t), W)$ of the general fiber is quasiequivalent to $D^b\mathrm{Coh}(V)$. Moreover, the action of the monodromy $\mu$ on $(W_{Y}\inv(t), W)$ is mirror to tensoring by the line bundle $\scr{O}(-H)$, and the natural transformation $\mu \to \mathrm{id}$ is mirror to the defining section $s: \scr{O}(-H) \to \scr{O}_V$.
\end{theorem*}

Hence, an LG version of our definition allows us to conclude that $D^{\pi}\scr{W}(W_{Y}\inv(0), W)$ is equivalent to $D^b \mathrm{Coh}(V\setminus H)$ using Lemma \ref{lemma:complement}. Moreover, using Definition \ref{defn:ci} it should be possible to generalize this to the case of complete intersections in toric varieties considered in the sequel to \cite{AA}. 

\subsection{Applications to Curves}\label{sec:curves}

Definition \ref{defn:auroux} is often quite straightforward to compute in practice, as we can see from a simple example.

\begin{example}
Consider the standard Lefschetz fibration $X = \CC^2$ and $f = x y$; let us directly compute the Fukaya category $\scr{W}(\set{xy=0})$ of the nodal conic. By Lemma \ref{lemma:split}, to find $\scr{W}(\set{xy=0})$, we just need to quotient $\scr{W}(\set{xy=t})$ by the image of $\cap$. But in this case, the only object to which we can apply $\cap$ is the standard thimble of $(\CC^2, xy)$, and $\cap$ simply gives the vanishing cycle $V \subset \set{xy=t}$. Hence $\scr{W}(\set{xy=0}) = \scr{W}(\set{xy=t})/\langle V \rangle$. Under mirror symmetry of $D^{\pi}\scr{W}(\set{xy=t})$ with $D^b\mathrm{Coh}(\CC^{\ast})$, $V$ corresponds to the skyscraper sheaf of a point $p \neq 0$, and the quotient $D^b\mathrm{Coh}(\CC^{\ast})/ \langle \scr{O}_p \rangle$ simply gives $D^b\mathrm{Coh}(\CC^{\ast}\setminus\set{p})$, which is quasiequivalent to the coherent sheaves on the pair of pants $D^b \mathrm{Coh}(\Pi_1)$, as expected from the above. \QEDB
\end{example}

In this case, we may furthermore use the explicit formulas for the localization of $A_{\infty}$ categories in \cite{LO} to compute the full $A_{\infty}$ structure of $\scr{W}(\set{xy=0})$. We shall have more to say about this in future work. 


Now, we consider the case of genus-$1$ curves with singularities. Consider the map $f: X \to \CC$ given by the Tate family of elliptic curves, with $f\inv(0)$ being an elliptic curve with a single node. We know by \cite{elliptic, LP2} and others that the Fukaya category of the general fiber $\mathscr{F}(f\inv(t))$ is derived equivalent to the category of coherent sheaves $\mathrm{Coh}(E)$ on a mirror elliptic curve $E$. In this case, the monodromy around this singularity is given by the negative Dehn twist around the corresponding vanishing cycle, and takes the Lagrangian $L \subset f\inv(t)$, mirror to $\scr{O}_E$, to the slope $(-1)$ Lagrangian $\mu(L) \subset f\inv(t)$, mirror to a degree $(-1)$ line bundle $\scr{L}\inv$ on $E$ \cite{elliptic}. Since the monodromy functor is given by translating fiberwise by $\mu(L)$, under this mirror equivalence $\mu$ is mirror to tensoring by this line bundle $\scr{L}\inv$ (see \S \ref{sec:gs} below for a discussion). Hence the natural transformation $\mu \to \mathrm{id}$ is mirror to a morphism $\scr{L}\inv \to \scr{O}_E$, that is, a section $s$ of the dual line bundle $\scr{L}$. Hence we may compare the localization of $\scr{W}(f\inv(t))$ at this natural transformation with that of $\mathrm{Coh}(E)$. On one hand, by Definition \ref{defn:auroux}, this localization of $\mathscr{F}(f\inv(t))$ yields $\mathscr{F}(f\inv(0))$, the Fukaya category of the nodal curve. On the mirror, by Lemma \ref{lemma:complement}, localization of $D^b \mathrm{Coh}(E)$ at the natural transformation $\scr{L}\inv \otimes \to \mathrm{id}$ given by multiplication by a section $s$ gives exactly the derived category of coherent sheaves of the complement, $D^b \mathrm{Coh}(E\setminus s\inv(0))$. Hence, we have a mirror equivalence between the wrapped Fukaya category of the nodal elliptic curve, and the derived category of coherent sheaves of a once-punctured elliptic curve. This argument may be generalized to show that the Fukaya categories of elliptic curves with more nodes are derived equivalent to elliptic curves with the corresponding number of punctures, by considering the monodromy of the $n$-nodal degeneration of elliptic curves considered by Gross-Siebert \cite{Chain}.

This can be generalized to the case of punctured elliptic curves with nodes by considering $f: X \to \CC$ the affine Tate family of $n$-punctured elliptic curves, with $f\inv(0)$ being an elliptic curve with a single node and $n$ punctures. Using the mirror equivalence of \cite{LP2, Lekili} for the general fiber, we can obtain mirror derived equivalences between wrapped Fukaya categories of elliptic curves with $n$ punctures and $m$ nodes, and derived categories of coherent sheaves of elliptic curves with $m$ punctures and $n$ nodes (using a stronger version of Lemma \ref{lemma:complement}). In this case, since $f$ satisfies the hypotheses of Theorem \ref{thm:main_thm} we have an additional mirror equivalence between the  coherent sheaves of elliptic curves with $m$ punctures and $n$ nodes with the Fukaya-Seidel category of an LG mirror $(X \times \CC, z f)$.

\subsection{Generalizations}\label{sec:gs}

The above examples for elliptic curves are special cases of a more general relationship between large complex structure limits and homological mirror symmetry for complex tori. Let us begin by recalling the Gross-Siebert setup in the context of toric degenerations of abelian varieties (see \cite[\S 2]{theta}). 

Let $M$ be a rank $n$ lattice, let $N = \mathrm{Hom}_{\ZZ}(M, \ZZ)$ be the dual lattice, and let $\Gamma \subset M$ be a rank-$n$ sublattice. The quotient $B = M_{\RR}/\Gamma$ naturally has the structure of an integral affine manifold. Suppose $\mathscr{P}$ is a $\Gamma$-periodic polyhedral decomposition of $M_{\RR}$ and $\varphi:M_{\RR} \to \RR$ be a strictly convex piecewise-linear function with integral slopes such that  for any $\gamma \in \Gamma$,
\begin{equation*}
    \varphi(m + \gamma) = \varphi(m) + \alpha_{\gamma}(m) 
\end{equation*}
where $\alpha_{\gamma}:M_{\RR} \to \RR$ is some integral affine function. Let $\check{\Gamma}$ be the lattice $\set{ \dd \alpha_{\gamma}(m): \; \; \gamma \in \Gamma}$ in $N_{\RR}$; then the quotient $\check{B} = N_{\RR}/\check{\Gamma}$ also has the structure of an integral affine manifold. The discrete Legendre transform (as described in \cite{logi}) yields $\check{\mathscr{P}}$ a dual $\check{\Gamma}$-periodic polyhedral decomposition of $N_{\RR}$ and a strictly-convex piecewise-linear function $\check{\varphi}: N_{\RR} \to \RR$. Let $X(B)$ be the quotient $TB/T^{\ZZ}B$ of the tangent bundle by the subbundle of integral vectors, and let $X^{\ast}(B)$ be the quotient $T^{\ast}B/T^{\ast}_{\ZZ}B$ of the cotangent bundle by the subbundle of integral covectors (and likewise for $\check{B}$). 

The convex function $\varphi$ induces an ample line bundle $\scr{L}_{\varphi}$ on $X(B)$ via a version of a construction of Mumford (\cite{mumford}, cf. \cite[\S 2]{theta}). There is a natural complex structure on $X(B)$ and this ample line bundle makes $X(B)$ into a K\"ahler manifold, symplectomorphic to the natural symplectic structure on $X^{\ast}(\check{B})$ coming from the canonical symplectic form on the cotangent bundle $T^{\ast}\check{B}$. We can construct $\scr{L}_{\varphi}$ explicitly as a quotient $M_{\RR} \times \CC/\Gamma$, where the action of $\gamma \in \Gamma$ is given by:
\begin{equation*}
    \gamma \cdot (m, z) = (m + \gamma, z \; \e{\dd \alpha_{\gamma}(m) + c_{\gamma}})
\end{equation*}
where $c_{\gamma}$ is the difference $\alpha_{\gamma}(m) - \dd \alpha_{\gamma}(m)$. This construction of Mumford yields toric degenerations $\scr{X}_{\varphi} \to D$ of $X(B)$ and $\check{\scr{X}}_{\check{\varphi}} \to D$ of $X(\check{B})$ over the unit disk $D \subset \mathbb{A}^1$ that are a Gross-Siebert mirror pair. A \textit{large volume limit} is the complement $X(B) \setminus s\inv(0)$ where $s$ is a section of the ample line bundle $\scr{L}_{\varphi}$, while the \textit{large complex structure limit} is the central fiber $\scr{X}_{\varphi}^0$ of the toric degeneration.

\newtheorem*

\begin{proof}
The Gross-Siebert construction of this toric degeneration in \cite{logi} has the property that the (counterclockwise) monodromy $\check{\mu}\inv: X(\check{B}) \to X(\check{B})$ around $0$ of the toric degeneration $\check{\scr{X}}_{\check{\varphi}} \to D$ is given by translation by the graph of the developing map of $\check{B}$ \cite[p.79]{logii}. The developing map is an integral affine immersion $\delta: \tilde{\check{B}} \to N_{\RR}$ from the integral affine universal cover $\tilde{\check{B}}$ to $N_{\RR}$, so that its graph $\sigma_1 \subset \tilde{\check{B}} \times N_{\RR}$ is given in local integral affine coordinates by $(y, y)$ \cite[p.13]{logi}. Under the Legendre transform, the developing map for $\check{B}$ is induced by the differential $\dd\varphi: \tilde{B} \to N_{\RR}$ of the multivalued convex function $\varphi$ on $B$ \cite[p.14]{logi}. Under the identification of $X(\check{B})$ with $X^{\ast}(B)$, the graph of the developing map for $\check{B}$ is given by the graph of $\dd \varphi$ in $T^{\ast}B$. 

Let $(F_m, \xi)$ denote the fiber of $X^{\ast}(B)$ over $[m] \in M_{\RR}$ equipped with a $U(1)$-local system $\xi$, as an object of the Fukaya category. The pair $(m,\xi)$ determines a point $(m, \check{\xi})$ in the mirror $X(B)$, where $\check{\xi}$ is the dual local system. The family Floer functor of Abouzaid \cite{FFT} sends the zero section in $X^{\ast}(B)$ to the structure sheaf $\scr{O}_{X(B)}$ and sections $L$ of $X^{\ast}(B)$ to line bundles $\scr{L}$ on $X(B)$. Under the family Floer functor, $\Gamma_{\dd \varphi}$ is taken to a line bundle $\scr{L}$ on $X(B)$, where the Floer complex $\mathrm{CF}((F_{m},\xi), \Gamma_{\dd \varphi})$ gives the stalk at $(m, \check{\xi}) \in X(B)$. Since $\Gamma_{\dd \varphi}$ is the graph of a section, the intersection $\Gamma_{\dd \varphi} \cap F_m$ consists of a single point $x_m$. Under the family Floer construction, the stalks over $(m, \check{\xi})$ and $(m + \gamma, \check{\xi})$ are identified via the change-of-basepoint map that takes
\begin{equation*}
    x_m \mapsto x_{m + \gamma}\; \e{\int_{u_{\gamma}} \omega}
\end{equation*}
where $u_{\gamma}$ represents the holomorphic polygon bounded by the zero section, the fibers $F_{m}, F_{m + \gamma}$, and the section $\Gamma_{\dd \varphi}$, as described in \cite[\S 2]{theta}. Since $X^{\ast}(B)$ carries a symplectic structure induced by the canonical $1$-form on $T^{\ast}B$, the area under the graph $\Gamma_{\dd \varphi}$ is given simply by 
\begin{equation*}
   \int_{u_{\gamma}} \omega =  \int_{[m, m + \gamma]} \dd \varphi = \varphi(m + \gamma) - \varphi(m) = \alpha_{\gamma}(m) = \dd \alpha_{\gamma}(m) + c_{\gamma}
\end{equation*}
which are exactly the transition functions for the fibers of $\scr{L}_{\varphi}$. Therefore the family Floer functor takes $\Gamma_{\dd \phi}$ to $\scr{L}_{\varphi}$. Moreover, since translating a Lagrangian by a section of $X^{\ast}(B)$ corresponds under the family Floer functor to tensoring by the corresponding line bundle $\scr{L}$, the functor also takes $\Gamma_{k \;\dd \phi}$ to $\scr{L}_{\varphi}^{\otimes k}$.

Abouzaid in \cite{FFT} proves that the family Floer functor gives an equivalence of $\mathscr{F}(X^{\ast}(B))$ with a full subcategory $\scr{C}$ of $D^b \mathrm{Coh}(X(B))$. We have seen that this subcategory $\scr{C}$ contains $\scr{L}_{\varphi}$ and all of its powers. By a theorem of Orlov \cite[Theorem 4]{orlov_generation}, powers of $\scr{L}_{\varphi}$ split-generate $D^b \mathrm{Coh}(X(B))$ and so $D^b \mathrm{Coh}(X(B))$ is the split-closure of $\scr{C}$. Therefore the split-closure of $\mathscr{F}(X(\check{B}))$ is equivalent to $D^b \mathrm{Coh}(X(B))$ under the family Floer functor. Since the monodromy $\hat{\mu}$ is given by translation by a section $\Gamma_{\dd \varphi}$ that is sent to the line bundle $\scr{L}_{\varphi}$, the family Floer functor sends the monodromy functor $\hat{\mu}\inv$ to tensoring by $\scr{L}_{\varphi}$. 

Therefore, Seidel's natural transformation $\hat{\mu} \to \mathrm{id}$ coming from monodromy around the large complex structure limit corresponds under our choice of homological mirror symmetry equivalence to a natural transformation $\scr{L}_{\varphi}^{-1}\otimes \to \mathrm{id}$ given by multiplication by some section $s: \scr{O}_{X(B)} \to \scr{L}_{\varphi}$. Now we can compare the localizations on both sides: by Lemma \ref{lemma:complement}, the localization of $D^b\mathrm{Coh}(X(B))$ at this natural transformation gives the coherent sheaves on the complement, $D^b \mathrm{Coh}(X(B)\setminus s\inv(0))$. Since $s$ is a section of $\scr{L}_{\varphi}$, the divisor $s\inv(0)$ is dual to the K\"ahler form $\omega_{X(B)}$, and so $X(B)\setminus s\inv(0)$ is a large-volume limit along $\omega_{X(B)}$. Since the total space of the large complex structure limit family is smooth \cite[p.26]{Chain}, on the A-side, by Definition \ref{defn:auroux}, we know that the localization of $\mathscr{F}(X(\check{B}))$ at the natural transformation given by monodromy yields $\mathscr{F}(\check{\scr{X}}_{\check{\varphi}}^0)$, the Fukaya category of the singular central fiber. Therefore we have a homological mirror symmetry quasiequivalence:
\begin{equation*}
    D^{\pi} \mathscr{F}(\check{\scr{X}}_{\check{\varphi}}^0) \simeq D^b \mathrm{Coh}(\check{X}\setminus s\inv(0))
\end{equation*}
between the large complex structure limit and the large volume limit.
\end{proof}

Following the last steps of the proof above, we also have the conditional result:

\lcsl*


A homological mirror symmetry equivalence of the form required to apply Theorem \ref{thm:lcsl} is expected to follow from combining Abouzaid's family Floer theory with ideas from the Gross-Siebert program (cf. \cite{theta}) in the same manner in the case where the Lagrangian torus fibration has singularities. Under the map $ \tilde{B} \times M_{\RR} \mapsto TB/T^{\ZZ} B$, the developing map is sent to the section:
\begin{equation*}
    y \mapsto   \sum_{i=1}^{n} y_i \pd{}{y_i}
\end{equation*}
which is exactly the canonical section $\sigma_1$ described in \cite{theta}. In \cite[\S 4]{theta}, it is explained that the family Floer functor should take the graph of $\sigma_1$ to the (inverse of) the ample line bundle $\scr{L}_{\varphi}$ giving the K\"ahler form on the mirror. Therefore a result similar to Theorem \ref{thm:new_theorem} should also hold in this case with a suitable modification of Definition \ref{defn:auroux}.

\nocite{*}
\bibliography{biblio}{}

\newcommand{\etalchar}[1]{$^{#1}$}
\providecommand{\bysame}{\leavevmode\hbox to3em{\hrulefill}\thinspace}
\providecommand{\MR}{\relax\ifhmode\unskip\space\fi MR }
\providecommand{\MRhref}[2]{%
  \href{http://www.ams.org/mathscinet-getitem?mr=#1}{#2}
}
\providecommand{\href}[2]{#2}
\begin{thebibliography}{AAE{\etalchar{+}}13}

\bibitem[AA21]{AA}
Mohammed Abouzaid and Denis Auroux, \emph{Homological mirror symmetry for
  hypersurfaces in $(\mathbb{C}^*)^n$}, 2021, preprint:
  \href{https://arxiv.org/abs/2111.06543}{arXiv:2111.06543}.

\bibitem[AAE{\etalchar{+}}13]{AAEKO}
Mohammed Abouzaid, Denis Auroux, Alexander~I. Efimov, Ludmil Katzarkov, and
  Dmitri Orlov, \emph{Homological mirror symmetry for punctured spheres}, J.
  Amer. Math. Soc. \textbf{26} (2013), no.~4, 1051--1083. \MR{3073884}

\bibitem[AAK16]{AAK}
Mohammed Abouzaid, Denis Auroux, and Ludmil Katzarkov, \emph{Lagrangian
  fibrations on blowups of toric varieties and mirror symmetry for
  hypersurfaces}, Publ. Math. Inst. Hautes \'{E}tudes Sci. \textbf{123} (2016),
  199--282. \MR{3502098}

\bibitem[Abo17]{FFT}
Mohammed Abouzaid, \emph{The family {F}loer functor is faithful}, J. Eur. Math.
  Soc. (JEMS) \textbf{19} (2017), no.~7, 2139--2217. \MR{3656481}

\bibitem[AG]{AG}
Mohammed Abouzaid and Sheel Ganatra, \emph{Generating {F}ukaya categories of
  {LG} models}, to appear.

\bibitem[AS]{AS}
Mohammed Abouzaid and Paul Seidel, \emph{Lefschetz fibration methods in wrapped
  {F}loer theory}, to appear.

\bibitem[AS15]{abouzaidSmith}
Mohammed Abouzaid and Ivan Smith, \emph{{K}hovanov homology from {F}loer
  cohomology}, preprint:
  \href{https://arxiv.org/abs/1504.01230}{arXiv:1504.01230}, 2015.

\bibitem[Aur18]{Speculations}
Denis Auroux, \emph{Speculations on homological mirror symmetry for
  hypersurfaces in {$(\mathbb{C}^\ast)^n$}}, Surveys in differential geometry
  2017. {C}elebrating the 50th anniversary of the {J}ournal of {D}ifferential
  {G}eometry, Surv. Differ. Geom., vol.~22, Int. Press, Somerville, MA, 2018,
  pp.~1--47. \MR{3838112}

\bibitem[BC17]{Cob}
Paul Biran and Octav Cornea, \emph{Cone-decompositions of {L}agrangian
  cobordisms in {L}efschetz fibrations}, Selecta Math. (N.S.) \textbf{23}
  (2017), no.~4, 2635--2704. \MR{3703462}

\bibitem[CE12]{CE}
Kai Cieliebak and Yakov Eliashberg, \emph{From {S}tein to {W}einstein and
  back}, American Mathematical Society Colloquium Publications, vol.~59,
  American Mathematical Society, Providence, RI, 2012, Symplectic geometry of
  affine complex manifolds. \MR{3012475}

\bibitem[FSS08]{FSS}
Kenji Fukaya, Paul Seidel, and Ivan Smith, \emph{Exact {L}agrangian
  submanifolds in simply-connected cotangent bundles}, Invent. Math.
  \textbf{172} (2008), no.~1, 1--27. \MR{2385665}

\bibitem[GPS19]{GPS2}
Sheel Ganatra, John Pardon, and Vivek Shende, \emph{Sectorial descent for
  wrapped {F}ukaya categories}, preprint:
  \href{https://arxiv.org/abs/1809.03427}{arxiv:1809.03427}, 2019.

\bibitem[GPS20]{GPS}
\bysame, \emph{Covariantly functorial wrapped {F}loer theory on {L}iouville
  sectors}, Publ. Math. Inst. Hautes \'{E}tudes Sci. \textbf{131} (2020),
  73--200. \MR{4106794}

\bibitem[GS03]{Chain}
Mark {Gross} and Bernd {Siebert}, \emph{{Affine manifolds, log structures, and
  mirror symmetry}}, {Turk. J. Math.} \textbf{27} (2003), no.~1, 33--60
  (English).

\bibitem[GS06]{logi}
Mark Gross and Bernd Siebert, \emph{Mirror symmetry via logarithmic
  degeneration data. {I}}, J. Differential Geom. \textbf{72} (2006), no.~2,
  169--338. \MR{2213573}

\bibitem[GS10]{logii}
\bysame, \emph{Mirror symmetry via logarithmic degeneration data, {II}}, J.
  Algebraic Geom. \textbf{19} (2010), no.~4, 679--780. \MR{2669728}

\bibitem[GS16]{theta}
\bysame, \emph{Theta functions and mirror symmetry}, Surveys in differential
  geometry 2016. {A}dvances in geometry and mathematical physics, Surv. Differ.
  Geom., vol.~21, Int. Press, Somerville, MA, 2016, pp.~95--138. \MR{3525095}

\bibitem[Hir17]{Hirano}
Yuki Hirano, \emph{Derived {K}n\"{o}rrer periodicity and {O}rlov's theorem for
  gauged {L}andau-{G}inzburg models}, Compos. Math. \textbf{153} (2017), no.~5,
  973--1007. \MR{3631231}

\bibitem[Joy15]{Joyce}
Dominic Joyce, \emph{A classical model for derived critical loci}, J.
  Differential Geom. \textbf{101} (2015), no.~2, 289--367. \MR{3399099}

\bibitem[KM98]{KM}
J\'{a}nos Koll\'{a}r and Shigefumi Mori, \emph{Birational geometry of algebraic
  varieties}, Cambridge Tracts in Mathematics, vol. 134, Cambridge University
  Press, Cambridge, 1998, With the collaboration of C. H. Clemens and A. Corti,
  Translated from the 1998 Japanese original. \MR{1658959}

\bibitem[KS14]{PervI}
Mikhail Kapranov and Vadim Schechtman, \emph{Perverse schobers}, preprint:
  \href{https://arxiv.org/abs/1411.2772}{arxiv:1411.2772}, 2014.

\bibitem[LO06]{LO}
Volodymyr {Lyubashenko} and Sergiy {Ovsienko}, \emph{{A construction of
  quotient $A_\infty$-categories.}}, {Homology Homotopy Appl.} \textbf{8}
  (2006), no.~2, 157--203 (English).

\bibitem[LP12]{LP2}
Yank{\i} Lekili and Timothy Perutz, \emph{{A}rithmetic mirror symmetry for the
  2-torus}, preprint: \href{https://arxiv.org/abs/1211.4632}{arXiv:1211.4632},
  2012.

\bibitem[LP17]{Lekili}
Yank{\i} Lekili and Alexander Polishchuk, \emph{Arithmetic mirror symmetry for
  genus 1 curves with {$n$} marked points}, Selecta Math. (N.S.) \textbf{23}
  (2017), no.~3, 1851--1907. \MR{3663596}

\bibitem[Mum72]{mumford}
David Mumford, \emph{An analytic construction of degenerating curves over
  complete local rings}, Compositio Math. \textbf{24} (1972), 129--174.
  \MR{352105}

\bibitem[Nad17]{NadlerII}
David Nadler, \emph{A combinatorial calculation of the {L}andau-{G}inzburg
  model {$M=\Bbb{C}^3$}, {$W=z_1z_2z_3$}}, Selecta Math. (N.S.) \textbf{23}
  (2017), no.~1, 519--532. \MR{3595901}

\bibitem[Nad19]{NadlerI}
\bysame, \emph{Mirror symmetry for the {L}andau-{G}inzburg {$A$}-model {$M=\Bbb
  C^n$}, {$W=z_1\cdots z_n$}}, Duke Math. J. \textbf{168} (2019), no.~1, 1--84.
  \MR{3909893}

\bibitem[Orl06]{Orlov}
D.~O. Orlov, \emph{Triangulated categories of singularities, and equivalences
  between {L}andau-{G}inzburg models}, Mat. Sb. \textbf{197} (2006), no.~12,
  117--132. \MR{2437083}

\bibitem[Orl09]{orlov_generation}
Dmitri Orlov, \emph{Remarks on generators and dimensions of triangulated
  categories}, Mosc. Math. J. \textbf{9} (2009), no.~1, 153--159, back matter.
  \MR{2567400}

\bibitem[Par12]{Parker}
Brett Parker, \emph{Log geometry and exploded manifolds}, Abh. Math. Semin.
  Univ. Hambg. \textbf{82} (2012), no.~1, 43--81. \MR{2922725}

\bibitem[PZ01]{elliptic}
Alexander {Polishchuk} and Eric {Zaslow}, \emph{{Categorical mirror symmetry in
  the elliptic curve.}}, {Proceedings of the winter school on mirror symmetry,
  Cambridge, MA, USA, January 1999}, Providence, RI: American Mathematical
  Society (AMS); Somerville, MA: International Press, 2001, pp.~275--295.

\bibitem[Sei98]{SeidelNotes}
Paul Seidel, unpublished notes, 1998.

\bibitem[Sei08a]{Subalgebras}
\bysame, \emph{{$A_\infty$}-subalgebras and natural transformations}, Homology
  Homotopy Appl. \textbf{10} (2008), no.~2, 83--114. \MR{2426130}

\bibitem[Sei08b]{SeidelBook}
\bysame, \emph{Fukaya categories and {P}icard-{L}efschetz theory}, Zurich
  Lectures in Advanced Mathematics, European Mathematical Society (EMS),
  Z\"urich, 2008.

\bibitem[Sei09a]{suspending}
\bysame, \emph{Suspending {L}efschetz fibrations, with an application to local
  mirror symmetry}, preprint:
  \href{https://arxiv.org/abs/0907.2063}{arXiv:0907.2063}, 2009.

\bibitem[Sei09b]{SHasHH}
\bysame, \emph{Symplectic homology as {H}ochschild homology}, Algebraic
  geometry---{S}eattle 2005. {P}art 1, Proc. Sympos. Pure Math., vol.~80, Amer.
  Math. Soc., Providence, RI, 2009, pp.~415--434. \MR{2483942}

\bibitem[Sei12]{LFI}
\bysame, \emph{Fukaya {$A_\infty$}-structures associated to {L}efschetz
  fibrations. {I}}, J. Symplectic Geom. \textbf{10} (2012), no.~3, 325--388.
  \MR{2983434}

\bibitem[Sei17]{LFII}
\bysame, \emph{Fukaya {$A_\infty$}-structures associated to {L}efschetz
  fibrations. {II}}, Algebra, geometry, and physics in the 21st century, Progr.
  Math., vol. 324, Birkh\"auser/Springer, Cham, 2017, pp.~295--364.
  \MR{3727564}

\bibitem[Spo02]{Malgrange}
Stanis{\l}aw Spodzieja, \emph{{\L}ojasiewicz inequalities at infinity for the
  gradient of a polynomial}, Bull. Polish Acad. Sci. Math. \textbf{50} (2002),
  no.~3, 273--281. \MR{1948075}

\bibitem[Syl19a]{Sylvan}
Zachary Sylvan, \emph{On partially wrapped {F}ukaya categories}, J. Topol.
  \textbf{12} (2019), no.~2, 372--441. \MR{3911570}

\bibitem[Syl19b]{Sylvan2}
\bysame, \emph{{O}rlov and {V}iterbo functors in partially wrapped {F}ukaya
  categories}, preprint:
  \href{https://arxiv.org/abs/1908.02317}{arXiv:1908.02317}, 2019.

\end{thebibliography}
\bibliographystyle{amsalpha}

\end{document}